\documentclass[12pt]{article}
\usepackage[usenames,dvipsnames,svgnames,table]{xcolor} 
\usepackage{kotex}
\usepackage{graphicx}
\usepackage{epsfig}
\usepackage{amssymb, amsmath, mathtools}
\usepackage{verbatim}
\usepackage[semicolon]{natbib}
\usepackage[affil-it]{authblk}
\usepackage{mathrsfs}
\usepackage{here}
\usepackage{makecell}
\usepackage{tikz}
\usepackage{enumerate}
\usepackage[shortlabels]{enumitem}
\usepackage{yfonts}
\usepackage{comment}
\usepackage{longtable}
\usepackage{tikz,amsmath, amssymb,bm,color}
\usetikzlibrary{circuits.logic.US,circuits.logic.IEC,fit}
\usepackage{algorithm2e}

\usepackage{amsthm}

\usepackage[colorlinks]{hyperref} 
\usepackage[colorinlistoftodos, textsize=scriptsize]{todonotes} 

\usepackage[framemethod=TikZ]{mdframed}
\mdfdefinestyle{JL}{%
    linecolor=blue,
    outerlinewidth=1pt,
    roundcorner=2pt,
    innertopmargin=0.1\baselineskip,
    innerbottommargin=0.1\baselineskip,
    innerrightmargin=2pt,
    innerleftmargin=2pt,
    backgroundcolor=orange!100!white}

\hypersetup{
  colorlinks,
  citecolor=Black,
  linkcolor=Red,
  urlcolor=Blue}

\includecomment{note} 

\newtheorem{theorem}{Theorem}[section]
\newtheorem{lemma}[theorem]{Lemma}
\newtheorem{definition}[theorem]{Definition}
\newtheorem{proposition}[theorem]{Proposition}
\newtheorem{corollary}[theorem]{Corollary}

\newtheorem{remark}[theorem]{Remark}

\DeclareMathOperator{\argmin}{argmin}

\everymath{\displaystyle}

\newcommand{\lra}{\longrightarrow}

\newcommand{\se}[1]{\section{#1}}
\newcommand{\sse}[1]{\subsection{#1}}

\newcommand{\be}{\begin{equation}}
\newcommand{\ee}{\end{equation}}
\newcommand{\bea}{\begin{eqnarray*}}
\newcommand{\eea}{\end{eqnarray*}}
\newcommand{\bean}{\begin{eqnarray}}
\newcommand{\eean}{\end{eqnarray}}
\newcommand{\ben}{\begin{enumerate}}
\newcommand{\een}{\end{enumerate}}
\newcommand{\bi}{\begin{itemize}}
\newcommand{\ei}{\end{itemize}}
\newcommand{\brem}{\begin{remark}}
\newcommand{\erem}{\end{remark}}
\newcommand{\bcen}{\begin{center}}
\newcommand{\ecen}{\end{center}}
\newcommand{\bsv}{\begin{semiverbatim}}
\newcommand{\esv}{\end{semiverbatim}}

\newcommand{\bt}{\begin{theorem}}
\newcommand{\et}{\end{theorem}}
\newcommand{\bl}{\begin{lemma}}
\newcommand{\el}{\end{lemma}}
\newcommand{\bd}{\begin{definition}}
\newcommand{\ed}{\end{definition}}
\newcommand{\bc}{\begin{corollary}}
\newcommand{\ec}{\end{corollary}}
\newcommand{\bp}{\begin{proposition}}
\newcommand{\ep}{\end{proposition}}

\newcommand{\bbE}{{ \mathbb{E}}}

\newcommand{\bbR}{ \mathbb{R}}

\newcommand{\bbZ}{ \mathbb{Z}}

\newcommand{\calC}{\mathcal{C}}
\newcommand{\calF}{\mathcal{F}}
\newcommand{\calL}{\mathcal{L}}
\newcommand{\calR}{\mathcal{R}}

\usepackage{xr}
\makeatletter
\newcommand*{\addFileDependency}[1]{
  \typeout{(#1)}
  \@addtofilelist{#1}
  \IfFileExists{#1}{}{\typeout{No file #1.}}
}
\makeatother

\title{Estimation of Conditional Mean Operator under the Bandable Covariance Structure}
\author[1]{Kwangmin Lee}
\author[2]{Kyoungjae Lee}
\author[1]{Jaeyong Lee}
\affil[1]{Department of Statistics, Seoul National University}
\affil[2]{Department of Statistics, Inha University}

\begin{document}

\maketitle

\begin{abstract}
	We consider high-dimensional multivariate linear regression models, where the joint distribution of covariates and response variables is a multivariate normal distribution with a bandable covariance matrix.
	The main goal of this paper is to estimate the regression coefficient matrix, which is a function of the bandable covariance matrix. 
	Although the tapering estimator of covariance has the minimax optimal convergence rate for the class of bandable covariances, 
	we show that it has a sub-optimal convergence rate for the regression coefficient; that is, a minimax estimator for the class of bandable covariances may not be a minimax estimator for its functionals.
	We propose the blockwise tapering estimator of the regression coefficient, which has the minimax optimal convergence rate for the regression coefficient under the bandable covariance assumption.
	We also propose a Bayesian procedure called the blockwise tapering post-processed posterior of the regression coefficient and show that the proposed Bayesian procedure has the minimax optimal convergence rate for the regression coefficient under the bandable covariance assumption.
	We show that the proposed methods outperform the existing methods via numerical studies.
	
\end{abstract}

\section{Introduction}

Consider the multivariate linear regression model 
\bea
Y_i = C X_i + \epsilon_i, \quad i =1,\ldots, n,
\eea
where $Y_i\in\bbR^q$ is a response vector, $X_i\in\bbR^{p_0}$ is a covariate vector,  $C \in \bbR^{q \times p_0}$ is a regression coefficient matrix, and 
$\epsilon_i \in \bbR^q, i=1,2, \ldots, n,$ are independent and identically distributed error vectors from a $q$-dimensional normal distribution with mean zero.  
The multivariate linear regression model has been used for various fields of applications. 
For example, \cite{zhao2018short} analyzed atmospheric data using the model to forecast PM$2.5$ concentration, and \cite{qian2020large} used the model to analyze the genomics data.

For the estimation of the multivariate linear regression coefficient $C\in \bbR^{q \times p_0}$, one of the most commonly used approaches is a penalized least square method, which finds the minimizer of the following objective function,
\bea
f(C) \,=\, \sum_{i=1}^n||Y_i -C X_i||_F^2 + P(C)   ,
\eea
where $P(C)$ is a penalty term.
The penalized least square method penalizes the objective function when the estimate $C$ deviates from the low dimensional structure which the true coefficient matrix $C$ is assumed to have, and it is especially useful under high-dimensional settings, where $p_0$ and $q$ can grow to infinity as $n\to \infty$. 
Various penalized methods for the multivariate linear regression model have been suggested \citep{chen2012sparse,chen2013reduced,uematsu2019sofar}.

Employing a covariance estimation method is another approach for the estimation of the regression coefficient.
The coefficient matrix $C$ can be considered as a function of the joint covariance matrix of the covariate vector $X\in\bbR^{p_0}$ and the response vector $Y\in\bbR^q$. 
Assume that $Z=(X^T , Y^T)^T \in \bbR^{p_0 +q}$ follows a joint distribution with a mean vector $\mu$ and a covariance matrix $\Sigma$ such that
\bea
\mu &=& \begin{pmatrix} \mu_X \\ \mu_Y \end{pmatrix} \\
\Sigma &=& 
\begin{pmatrix}
	\Sigma_{XX} & \Sigma_{XY}\\
	\Sigma_{YX} & \Sigma_{YY}
\end{pmatrix},
\eea
where $\mu_X\in\bbR^{p_0}$, $\mu_Y\in\bbR^{q}$, $\Sigma_{XX}\in \bbR^{p_0\times p_0}$ and $\Sigma_{YY}\in\bbR^{q\times q}$.
Then, we have 
\bea
(\mu_0 , \Psi_0) &:=&  (\mu_Y-\Sigma_{YX}\Sigma_{XX}^{-1}\mu_X, \,\, \Sigma_{YX}\Sigma_{XX}^{-1}) \\
&=& \argmin_{(\mu, \Psi)}  \bbE \{ (Y-\mu-\Psi X) (Y-\mu-\Psi X)^T \},
\eea
and $\mu_0 + \Psi_0 x$ is the conditional mean of $Y$ given $X= x$ if $Z$ follows the multivariate Gaussian distribution. 
Note that $\mu_0$ is the zero vector if we assume that $\mu_X$ and $\mu_Y$ are zero vectors. 
In this case, the coefficient matrix $C$ in the multivariate regression model corresponds to $\Psi_0 = \Sigma_{YX} \Sigma_{XX}^{-1}$, which is a function of the covariance matrix $\Sigma$ and is called {\it the conditional mean operator}.
Thus, covariance estimators can be used for the estimation of the conditional mean operator. 

We need to consider a high-dimensional covariance estimation method when we use a covariance estimator for the multivariate regression model under high-dimensional settings. 
Suppose $Z_1, Z_2, \ldots, Z_n$ are independent and identically generated from a $p$-dimensional distribution with mean zero and covariance matrix $\Sigma$. We refer to the estimation of covariance $\Sigma$ as high-dimensional covariance estimation when $p$ is assumed to go to infinity as $n\lra \infty$. 
Since traditional covariance estimation methods, such as the sample covariance matrix and the Bayesian method by the inverse-Wishart prior, are not consistent when $p$ is larger than $n$ \citep{johnstone2009consistency,lee2018optimal}, various structural assumptions on covariance matrices have been used to reduce the number of effective parameters. For example, the banded covariances \citep{lee2020post}, the bandable covariances  \citep{bickel2008regularized}, sparse covariances \citep{cai2013sparse} and sparse spiked covariances \citep{cai2015optimal} have been considered. 
These structural assumptions can be used in the joint covariance matrix of covariates and response variables when we employ covariance estimation for the multivariate regression under the high-dimensional settings.

In this paper, we consider the multivariate linear regression model, where the covariate vector and the response vector jointly follow a multivariate normal distribution with a bandable covariance matrix.
Under the bandable covariance assumption, the farther apart two variables are, the smaller their covariance is. 
On the frequentist side, \cite{cai2010optimal,cai2012minimax} proved that {\it the tapering estimator of covariance} has the minimax optimal convergence rates for the class of bandable covariances under the spectral norm, Frobenius norm, and matrix $l_1$ norm. 
Therefore, a naive approach would be estimating the conditional mean operator based on the tapering estimator of covariance (or other minimax covariance estimators).

Unfortunately, even if a covariance estimator $\hat{\Sigma}$ has the minimax optimal convergence rate for the covariance $\Sigma$, it does not imply that $f(\hat{\Sigma})$ has also the minimax optimal convergence rate for $f(\Sigma)$ where $f$ is a function on the space of covariances.
Thus, the estimator for $\Sigma_{YX} \Sigma_{XX}^{-1}$ based on the tapering estimator of covariance may not have the minimax optimal convergence rate.
Furthermore, there is no Bayesian method achieving the minimax posterior convergence rate for the class of bandable covariances.
Note that \cite{silva2009hidden}, \cite{khare2011wishart} and \cite{lee2020post} proposed Bayesian procedures for banded covariances, but the class of bandable covariances considered in this paper is larger than the class of banded covariances.

We investigate the decision-theoretic property of the tapering estimator when the parameter of interest is the conditional mean operator, $\Sigma_{YX}\Sigma_{XX}^{-1}$, instead of the covariance itself.
We define {\it the tapering estimator of regression coefficient} as the plug-in estimator, the tapering estimator of covariance plugged into the conditional mean operator, and show that the tapering estimator of regression coefficient has a sub-optimal convergence rate for $\Sigma_{YX}\Sigma_{XX}^{-1}$ under the bandable covariance assumption.
We propose a minimax optimal estimator for $\Sigma_{YX}\Sigma_{XX}^{-1}$ by modifying the tapering estimator of regression coefficient and call it {\it the blockwise tapering estimator of regression coefficient}.

As a Bayesian procedure for the conditional mean operator under the bandable covariance assumption, we propose post-processed posterior method \citep{lee2020post}.
A post-processed posterior \citep{lee2020post} is a posterior constructed by transforming posterior samples from the initial posterior, which is typically a computationally convenient posterior. 
This idea is especially useful when it is difficult to impose a prior distribution on a restricted parameter space due to an unknown normalizing constant.
For a given parameter space $\Theta^*$, suppose that we are interested in restricted parameter space, $\Theta \subset \Theta^*$.
A post-processed posterior can be obtained by generating samples from an initial posterior on $\Theta^*$ and post-processing the posterior samples so that the transformed post-processed samples belong to $\Theta$. When the post-processing function is a projection map from $\Theta^*$ to $\Theta$, the method is called the posterior projection method. The posterior projection method has been suggested for various settings including \cite{dunson2003bayesian}, \cite{gunn2005transformation}, \cite{lin2014bayesian} and \cite{chakraborty2020convergence}, and was investigated in general aspects by \cite{patra2018constrained}.
The idea of transforming posterior samples was also used for the inference on covariance or precision matrices in \cite{lee2020post} and \cite{bashir2018post}.

We suggest two post-processed posteriors for the conditional mean operator. Both methods use the inverse-Wishart distribution as the initial prior distribution on the unconstrained covariance matrix space and use the tapering function and the blockwise tapering function as the post-processing functions for the conditional mean operator $\Sigma_{YX}\Sigma_{XX}^{-1}$.  We present the asymptotic analysis to justify the proposed post-processed posteriors, and show that the post-processed posterior by the blockwise tapering function has the minimax optimal convergence rate.

The rest of the paper is organized as follows. In Section \ref{sec:CM_bandable}, we introduce the blockwise tapering estimator for the inference of the conditional mean operator under the bandable covariance assumption and show that this estimator has the minimax convergence rate. In Section \ref{sec:btppp}, we introduce the post-processed posteriors for the conditional mean operator, and present the posterior convergence rates.
Simulation studies and real data analysis are given in Section \ref{sec:numerical}. We conclude this paper with a discussion section.
The proofs of theorems that give the upper bound and lower bound of the convergence rate of the blockwise tapering estimator are given in Appendix \ref{sec:proof}, and the proofs of the other theorems and lemma are given in the supplementary material.

\se{Blockwise tapering estimator and minimax analysis}\label{sec:CM_bandable}

\sse{Notation}

Let $q$, $k$ and $l$ be positive integers with $l\vee k\le q$. 
For a $q\times q$-matrix $\Sigma$ and positive real numbers $\sigma_{ij}$, $1\le i,j\le q$, let $\Sigma= (\sigma_{ij})_{1\le i,j\le q} = (\sigma_{ij})$ when $\sigma_{ij}$ is equal to the $(i,j)$ element of $\Sigma$.
We define sub-matrix operators $M_l^{(k)} :\bbR^{q\times q} \mapsto \bbR^{k^*\times k^*}$, where $k^*=\{(l+k-1)\wedge q\}-(l\vee 1) +1$, and $M_l^{*(k)} :\bbR^{q\times q} \mapsto \bbR^{q\times q}$ as
\bea
M_l^{(k)}(\Sigma) &=& (\sigma_{ij})_{(l\vee 1) \le i,j \le \{(l+k-1) \wedge q\}}\\
M_l^{*(k)}(\Sigma) &=& (\sigma_{ij}I[(l \vee 1) \le i,j \le \{(l+k-1)\wedge q\} ])_{1\le i,j\le q},
\eea
for $\Sigma = (\sigma_{ij})_{1\le i,j\le q}$. 
Let $\Sigma_{a:b,c:d}$ be the sub-block matrix of $\Sigma\in\bbR^{q\times q}$ with $(a,a+1,\ldots,b-1,b)$ rows and $(c,c+1,\ldots,d-1,d)$ columns for positive integers $a,b,c$ and $d$ with $1\le a<b\le q$ and $1\le c< d\le q$.  We also let $X_{a:b} = (x_a,x_{a+1},\ldots,x_{b-1},x_b)\in\bbR^{b-a+1}$ for a vector $X = (x_1,x_2,\ldots,x_q)\in \bbR^q$ and positive integers $a$ and $b$ with $1\le a<b\le q$.

For a $q\times q$-matrix $\Sigma =(\sigma_{ij})_{1\le i,j\le q}$ and a positive integer $k$ with $k \le q$, 
define the tapering function $T_k(\Sigma)$, which was first defined in \cite{cai2010optimal}, as
\bea
T_k(\Sigma)=(w^{(k)}_{ij}\sigma_{ij})_{1\le i,j\le q},
\eea
where 
\[
w_{ij}^{(k)}= 
\begin{cases}
	1,& \text{when } |i-j|\le k/2\\
	2-\frac{|i-j|}{k/2},              & \text{when } k/2<|i-j|<k\\
	0, & \text{otherwise}
\end{cases}.
\] 

For any sequences $a_n$ and $b_n$ of positive real numbers, we denote $a_n=o(b_n)$ if $\lim\limits_{n\lra \infty}a_n/b_n=0$, and $a_n=O(b_n)$ if $\limsup\limits_{n\lra \infty}a_n/b_n=C$ for a positive constant $C$.
We denote $a_n \preceq b_n$ if $a_n\le C b_n$ for all sufficiently large $n$ and a positive constant $C$.

Let $||\Sigma || =||\Sigma ||_2= \{ \lambda_{\max}(\Sigma\Sigma^T) \}^{1/2}$ be the spectral norm of a covariance matrix $\Sigma$, where $\lambda_{\max}(\Sigma)$ is the maximum eigenvalue of $\Sigma$. Given positive integers $p$ and $p_0$ with $p_0<p$, let $A_{XX} =A_{(p_0+1):p,1:p_0}$ and $A_{YX} = A_{(p_0+1):p,1:p_0}$ for a positive $p\times p$-matrix $A$. We also let $\Sigma_{0,XX}$ and $\Sigma_{0,YX}$ denote $(\Sigma_0)_{XX}$ and $ (\Sigma_0)_{YX}$, respectively.

\sse{Blockwise Tapering Estimator} 
Let $n$, $p$ and $p_0$ be positive integers with $p_0< p$. 
Suppose  $Z_1,Z_2,\ldots,Z_n$ are independent and identically distributed from a $p$-variate Gaussian distribution with mean zero and covariance matrix $\Sigma_0$, which is denoted by  $N_p(0,\Sigma_0)$, where $Z_i = (X_i^T,Y_i^T)^T$, $X_i\in \bbR^{p_0}$ and $Y_i\in\bbR^{p-p_0}$ for $i\in\{1,2,\ldots,p\}$. When only the first $p_0$ elements of $Z_i$, i.e. $X_i$, are given, the conditional mean vector for the other $p - p_0$ variables is 
\bean
\Sigma_{0,YX}(\Sigma_{0,XX})^{-1} X_i\label{formula:CM}.
\eean 
The conditional mean operator  $ \Sigma_{0,YX}(\Sigma_{0,XX})^{-1}$ in \eqref{formula:CM}  is the estimand we focus on in this paper. 
We define the transformation $\psi$ from a covariance to the conditional mean operator as
\bea
\psi(\Sigma) :=\psi(\Sigma;p_0)= \Sigma_{YX}(\Sigma_{XX})^{-1} ,
\eea
for $\Sigma\in\calC_p$, where $\calC_p$ is the set of all $p\times p$-dimensional positive definite matrices.

We assume $\Sigma_0$ belongs to a class of bandable covariances, $\calF_\alpha$, which is defined as 
\bea
\calF_\alpha &:=& \calF_{p,\alpha}(M, M_0, M_1) \nonumber \\
&=& \Big\{ \Sigma = (\sigma_{ij})_{1\le i,j\le p} \in \calC_p: \sum_{ (i,j): |i-j| \geq k} |\sigma_{ij} | \leq M k^{-\alpha}, \forall k \geq 1, \lambda_{\max}(\Sigma) \leq   M_0 , \lambda_{\min}(\Sigma) \geq   M_1  \Big\},
\eea
for some positive constants $\alpha, M>0$ and $0<M_1 < M_0$, where $ \lambda_{min}(\Sigma)$ is the minimum eigenvalue of $\Sigma$.
\cite{bickel2008regularized} and \cite{cai2010optimal} also considered the same class of bandable covariances except the minimum eigenvalue condition.

A natural estimator for $\psi(\Sigma_0)$ is the plug-in estimator, the tapering estimator of covariance plugged into $\psi$, for the tapering estimator of covariance has the minimax optimal convergence rate for the class of bandable covariances under the spectral norm loss \citep{cai2010optimal}. For the positive-definiteness is necessary for the covariance estimator, we modify the tapering estimator of covariance so that it is positive-definite and call it  {\it adjusted tapering estimator of covariance}: 
\bea
T_k^{(\epsilon_n)}(S_n) := T_k(S_n) + ([\epsilon_n -\lambda_{\min}\{T_k(S_n)\}]\vee 0 )I_p,
\eea
where $\epsilon_n > 0$ is the positive-definite adjustment parameter, $S_n$ is the sample covariance matrix $\sum_{i=1}^n Z_i Z_i^T/n$, and $I_p$ is the $p\times p$ identity matrix. We call the plug-in estimator with adjusted tapering estimator of covariance {\it the tapering estimator of regression coefficient}, in short {\it the tapering estimator}.

Since every column vector in the tapering estimator is not the zero vector with probability one, the tapering estimator uses all variables in a given covariate vector when the estimator is used as the regression coefficient.
In other words, in the variable selection perspective, all variables are selected when the tapering estimator is used. 
Note that selecting out negligible covariates can increase the accuracy of a regression estimator, and partial correlations between covariates and responses have been used as a criterion for the variable selection \citep{li2017variable, buhlmann2010variable}. 
We find covariates which have weak partial correlations with the response variables under the bandable covariance assumption by investigating the elements in the inverse matrix of the covariance, called the precision matrix in Theorem \ref{theorem:bandableinv}.
\begin{theorem}\label{theorem:bandableinv}
	Suppose $\Sigma_0\in\calF_{p,\alpha}(M,M_0,M_1)$, and let $\Sigma_0^{-1}=(w_{ij})$. 
	There exist some positive constants $C$ and $\lambda$ depending only on $M_0$, $M_1$ and $M$ such that 
	\bea
	\max_j\sum_i \{ |w_{ij}|: |i-j|>ak\log k\} 
	&\le& C(k^{-a\lambda+1} + k^{-\alpha}),
	\eea
	for all $a>0$ and all sufficiently large integer $k$ with $p> k\vee (ak\log k)$.	 
\end{theorem}
See the supplementary material for the proof.
Note \cite{lauritzen1996graphical} showed that the partial correlation between variable $i$ and variable $j$, $\rho_{ij}$, is 
\bea
\rho_{ij} = \frac{w_{ij}}{\sqrt{w_{ii} w_{jj}}},
\eea
where $\Sigma_0^{-1}=(w_{ij})$.
Since $|w_{ii}| \ge M_0^{-1}$ for all $i\in\{1,2,\ldots,p\}$, each element in response vector $(Z_i)_{p_0+1:p}$ has negligibly weak partial correlations with remote covariates $(Z_i)_j$, i.e. variables with $|j-p_0|$ large and $j\le p_0$, when $k$ is sufficiently large by Theorem \ref{theorem:bandableinv}.
Thus, selecting out these negligible covariates could yield a more accurate estimator for the conditional mean operator.

Based on the above argument, we propose  {\it the blockwise tapering estimator of regression coefficient}, in short {\it the blockwise tapering estimator}. 
Let $\bbZ$ be the set of all integers and $\lfloor x\rfloor= \max\{ z \in \bbZ : z\le x\} $.
For positive real numbers $a$ and $\epsilon_n$, and a positive integer $k$ with $2\lfloor ak\log k\rfloor\le p_0$,
define  {\it the blockwise tapering estimator} as 
\bean
\phi(S_n;2 \lfloor ak\log k\rfloor,\epsilon_n) &:=&\phi(S_n;p_0,2 \lfloor ak\log k\rfloor,\epsilon_n)\nonumber \\
&=& T_k(S_n)_{YX} \Lambda^{(\epsilon_n)} \{T_k(S_n)_{XX} ; 2 \lfloor ak\log k\rfloor\}\label{formula:cmoperator},
\eean
where $\Lambda^{(\epsilon_n)}(A;b)$ is defined as for a $p_0\times p_0$ matrix $A$
\bea
\Lambda^{(\epsilon_n)}(A;b)=
\begin{pmatrix}
	O_{(p_0-b) \times (p_0-b)}  & O_{(p_0-b) \times b} \\
	O_{(b \times (p_0-b)}  & \{M_{p_0-b+1}^{(b)}(A) + ([\epsilon_n-\lambda_{\min}\{ M_{p_0-b+1}^{(b)}(A) \}]\vee 0) I_{b} \}^{-1}
\end{pmatrix},
\eea
where $O_{c\times d}$ is the $c\times d$-zero matrix for positive integers $c$ and $d$.
Given a covariate vector $x\in \bbR^{p_0}$, the blockwise tapering estimator uses
only $x_{(p_0-2 \lfloor ak\log k\rfloor+1):p_0}$.
Thus, the covariates which have weak partial correlations with response variables are not used. 

\sse{ Minimax Analysis of Blockwise Tapering Estimator}
We give the convergence rates of the tapering and blockwise tapering estimators and show that the blockwise tapering estimator has the minimax convergence rate.
We use the loss function on $\bbR^{(p-p_0)\times p_0}$ 
\bean
L\{\hat{C}, \psi(\Sigma_0)\}= || \hat{C}  - \psi(\Sigma_0)||_2,\label{formula:loss}
\eean  
for a pair of parameter $\psi(\Sigma_0)$ and estimator $\hat{C}$.
The loss function gives the upper bound of the estimation error of $E(Y\mid X=x)$ given $x\in\bbR^{p_0}$, because the definition of the operator norm gives 
\bea
|| \hat Cx - E(Y\mid X=x)||_2  &=& ||\{\hat C  -  \psi(\Sigma_0)\}x||_2 \\
&\le& L\{ \hat C,\psi(\Sigma_0)\} ||x||_2.
\eea  
We show that the tapering estimator has a sub-optimal convergence rate under the loss function \eqref{formula:loss}, while the blockwise tapering estimator has the minimax optimal convergence rate.

Theorem \ref{theorem:taperupper} gives the convergence rate of the tapering estimator.
If we set $\epsilon_n$ such that $ p^{1/2}5^{k/2} n\exp(-\lambda n) \preceq \epsilon_n^2 \preceq (k+\log p)/n$, then the convergence rate is $ (k+\log p)/n + k^{-2\alpha}$, which is the same rate as the convergence rate of the tapering estimator of covariance \citep{cai2010optimal}. 
\begin{theorem}\label{theorem:taperupper}
	Suppose $\Sigma_0\in\calF_{p,\alpha}(M,M_0,M_1)$.
	Let $k$ be a positive integer with $k<p_0$.
	If $k \vee \log p =o(n)$, $\epsilon_n=O(1)$ and $\lfloor k/2 \rfloor > \{4M/\lambda_{\min}(\Sigma_0)\}^{1/\alpha}$, then there exist some positive constants $C$ and $\lambda$ depending only on $M$, $M_0$, $M_1$ and $\alpha$ such that 
	\bea
	E_{\Sigma_0} ( || \psi(\Sigma_0)  -  \psi\{T_k^{(\epsilon_n)}(S_n)\}||^2 ) \le C\Big\{ k^{-2\alpha} + \frac{k+\log p}{n} +\epsilon_n^2 +\frac{p^{1/2} 5^{k/2} \exp(-\lambda n)}{\epsilon_n^2}  \Big\},
	\eea
	for all sufficienly large $n$.
\end{theorem}
The proof of this theorem is given in the supplementary material.

Next, we show the convergence rate of the blockwise tapering estimator.
The blockwise tapering estimator is designed to estimate $\phi(\Sigma_0;2\lfloor ak\log k\rfloor,0)$ which approximates $\psi(\Sigma_0)$.
Lemma \ref{lemma:trueerror} gives the approximation error, which is negligible when $k$ is large enough. 
Based on the approximation error, the convergence rate of the blockwise tapering estimator is given in Theorem \ref{theorem:frequpper}.
If we set $\epsilon_n$ such that $p^{1/2}5^{k/2} n\exp(-\lambda n) \preceq\epsilon_n^2\preceq k/n$, the convergence rate of the blockwise tapering estimator is $ k/n +k^{-2\{\alpha \wedge (a\tau-1)\}} $.
\begin{lemma}\label{lemma:trueerror}
	Suppose $\Sigma_0\in\calF_{p,\alpha}(M,M_0,M_1)$. 
	There exist some positive constants $C$ and $\tau$ depending only on $M$, $M_0$ and $M_1$ such that 
	\bea
	||\psi(\Sigma_0) - T_k(\Sigma_{0})_{YX} \Lambda^{(0)}\{T_k(\Sigma_{0,XX});2\lfloor ak\log k\rfloor \}|| \le C (k^{-\alpha} + k^{-a\tau+1}),
	\eea
	for all $a>0$ and all sufficiently large integers $k$ and $p_0$ with $\lfloor ak\log k\rfloor /2\ge k$ and $2\lfloor ak\log k\rfloor< p_0$.
\end{lemma}
The proof of this lemma is given in the supplementary material.
\begin{theorem}\label{theorem:frequpper}
	Suppose $\Sigma_0\in\calF_{p,\alpha}(M,M_0,M_1)$. 
	If $ k\vee \log p =o(n)$, $\lfloor k/2 \rfloor > \{4M/\lambda_{\min}(\Sigma_0)\}^{1/\alpha}$ and $\epsilon_n=O(1)$, then
	there exist some positive constants $C$, $\lambda$ and $\tau$ depending only on $M$, $M_0$, $M_1$ and $\alpha$ such that
	\bea
	&&E_{\Sigma_0} ( || \psi(\Sigma_0)  - \phi(S_n;2\lfloor ak\log k\rfloor,\epsilon_n)||^2 ) \\
	&\le& C\Big\{ k^{-2\{\alpha \wedge (a\tau-1)\}} + \frac{k}{n} +\epsilon_n^2 +\frac{p^{1/2} 5^{k/2} \exp(-\lambda n)}{\epsilon_n^2}  \Big\},
	\eea
	for all $a>0$ and all sufficiently large $n$, $k$ and $p_0$ with $\lfloor ak\log k\rfloor /2\ge k$ and $p_0>2\lfloor ak\log k\rfloor$.
\end{theorem}
The proof of this theorem is given in Appendix \ref{ssec:thm4}.

Next, we give the lower bound of the minimax risk for the conditional mean operator under the bandable covariance assumption to show that the blockwise tapering estimator is a minimax optimal estimator. 
Let $\hat{C}=\hat{C}(X_1,X_2,\ldots,X_n)$ be an estimator on $\bbR^{p-p_0 \times p_0}$. The minimax risk is defined as 
\bea
\inf_{\hat{C}}\sup_{\Sigma_0\in \calF_\alpha} E || \psi(\Sigma_0) -\hat{C} ||^2 .
\eea
Theorem \ref{theorem:condmeanlower} gives a lower bound of the minimax risk as $n^{-2\alpha/(2\alpha+1)}$.
If we set $k$, $a$ and $\epsilon_n$ of the blockwise tapering estimator such that
$k=n^{1/(2\alpha+1)}$, $a>(\alpha+1)/\tau$ and $p^{1/2}5^{k/2} n\exp(-\lambda n) \preceq\epsilon_n^2\preceq k/n$, then the convergence rate is the same as the lower bound asymptotically.
Thus, the minimax convergence rate is $n^{-2\alpha/(2\alpha+1)}$, and the blockwise tapering estimator attains the convergence rate.

\begin{theorem}\label{theorem:condmeanlower}
	There exist some positive constants $C$ and $\gamma$ depending only on $M$, $M_0$, $M_1$ and $\alpha$ such that
	\bea
	\inf_{\hat{C}}\sup_{\Sigma_0\in \calF_\alpha} E || \psi(\Sigma_0) -\hat{C} ||^2 \ge C  n^{-2\alpha/(2\alpha+1)},
	\eea
	for all sufficiently large $n$ and $p_0$ with $p_0>\gamma n^{1/(2\alpha +1)}$
\end{theorem}
See Appendix \ref{ssec:thm5} for the proof.

\se{Blockwise Tapering Post-Processed Posterior}\label{sec:btppp}

We propose the Bayesian counterparts of the tapering estimator and the blockwise tapering estimator using the post-processed posterior method. See \citet{lee2020post}. The algorithm for the post-processed posteriors consists of the following two steps.
\begin{itemize}
	\item[(a)] (Initial posterior sampling step) First, we obtain the initial conjugate posterior distribution on the unconstrained parameter space. We take the inverse-Wishart distribution $IW_p(B_0, \nu_0)$ as the initial prior distribution of which density function is 
	$$\pi^i(\Sigma) \propto |\Sigma|^{-\nu_0/2}e^{-tr(\Sigma^{-1}B_0)/2}, \quad \Sigma\in\calC_p,$$
	where $B_0\in\calC_p$ and $\nu_0>2p$. Then, the initial posterior distribution $\pi^i(\Sigma | \bbZ_n)$ is $IW_p(B_0+nS_n,\nu_0+n)$, where $n$ is the number of observations, $S_n= n^{-1}\sum_{i=1}^n Z_i Z_i^T$ and $\bbZ_n = (Z_1,\ldots, Z_n)$. We generate $\Sigma^{(1)},\Sigma^{(2)},\ldots,\Sigma^{(N)}$ from the initial posterior distribution.
	\item[(b)]   (Post-processing step)  Second, we post-process the samples from the initial posterior distribution with $\psi \{T_k^{(\epsilon_n)}(\cdot)\}$ or $\phi(\cdot;2\lfloor ak\log k\rfloor ,\epsilon_n)$, which are called {\it the tapering function} and {\it the blockwise tapering function}.   
	
\end{itemize}
We call the post-processed posteriors obtained from the post-processing functions
{\it the tapering post-processed posterior} (tapering PPP) and {\it the blockwise tapering post-processed posterior} (blockwise tapering PPP).



We use the decision-theoretic framework \citep{lee2018optimal,lee2020post} to prove the minimax optimality of the blockwise tapering post-processed posterior.
We define P-loss $\calL(\cdot,\cdot)$ and P-risk $\calR(\cdot,\cdot)$ for the conditional mean operator as 
\bea
\calL\{ \psi(\Sigma_0), \pi^{pp}(\cdot\mid\bbZ_n;f)    \} &:=&  E^{\pi^{i}} (||  \psi(\Sigma_0)- f(\Sigma) ||^2\mid \bbZ_n)\\
\calR\{\psi(\Sigma_0), (\pi^{i}, f)\} &:= & E_{\Sigma_0} \{  E^{\pi^i}  (||\psi(\Sigma_0)- f(\Sigma)||^2 \mid \bbZ_n)\},
\eea
where $ \pi^{pp}(\cdot\mid\bbZ_n;f)$ is the post-processed posterior distribution derived from  initial prior $\pi^i$ and post-processing function $f$, and $(\pi^{i}, f)$ is a pair of initial prior $\pi^i$ and post-processing function $f$.
Theorems \ref{theorem:TPPPupper} and \ref{theorem:PPPupper} give the P-risk convergence rates of the tapering and the blockwise tapering post-processed posteriors, respectively. The convergence rates are the same as their frequentist counterparts. 
\begin{theorem}\label{theorem:TPPPupper}
	Suppose $\Sigma_0\in\calF_{p,\alpha}(M,M_0,M_1)$. 
	Let $k$ be a positive integer with $k<p_0$, and let the prior $\pi^i$ of $\Sigma$ be $IW_p(A_n,\nu_n)$ for $A_n\in\calC_p$ and $\nu_n>2p$. If $\epsilon_n=O(1)$, $\lfloor k/2 \rfloor > \{4M/\lambda_{\min}(\Sigma_0)\}^{1/\alpha}$ and $k \vee ||A_n|| \vee (\nu_n-2p)\vee\log p=o(n) $, then there exist positive constants $C$ and $\lambda$  depending only on $M$, $M_0$ and $M_1$ such that 
	\bea
	&&E_{\Sigma_0}\{ E^{\pi^i} ( || \psi(\Sigma_0) - \psi\{ T_k^{(\epsilon_n)}(\Sigma)\} ||^2\mid\bbZ_n)\}\\
	&\le&  C\Big\{ k^{-2\alpha} + \frac{k+\log p}{n} +\epsilon_n^2 +\frac{p^{1/2} 5^{k/2} \exp(-\lambda n)}{\epsilon_n^2}  \Big\},
	\eea
	for all sufficiently large $n$ and $k$.
\end{theorem}
The proof of this theorem is given in the supplementary material.
\begin{theorem}\label{theorem:PPPupper}
	Suppose $\Sigma_0\in\calF_{p,\alpha}(M,M_0,M_1)$. 
	Let the prior $\pi^i$ of $\Sigma$ be $IW_p(A_n,\nu_n)$ for $A_n\in\calC_p$ and $\nu_n>2p$. If $\epsilon_n=O(1)$, $\lfloor k/2 \rfloor > \{4M/\lambda_{\min}(\Sigma_0)\}^{1/\alpha}$ and $k \vee ||A_n|| \vee (\nu_n-2p)\vee\log p=o(n) $, then there exist positive constants $C$, $\tau$ and $\lambda$  depending only on $M$, $M_0$ and $M_1$ such that  
	\bea
	&&E_{\Sigma_0}\{ E^{\pi^i} ( || \psi(\Sigma_0) - \phi(\Sigma;2\lfloor ak\log k\rfloor ,\epsilon_n) ||^2\mid\bbZ_n)\} \\
	&\le&  C\Big\{ k^{-2(\alpha \wedge (a\tau-1))} + \frac{k}{n} +\epsilon_n^2 +\frac{p^{1/2} 5^{k/2} \exp(-\lambda n)}{\epsilon_n^2}  \Big\},
	\eea
	then for all $a>0$ and all sufficiently large $n$, $k$ and $p_0$ with $\lfloor ak\log k\rfloor/2> k$ and $p_0>2\lfloor ak\log k\rfloor $.
\end{theorem}
The proof of this theorem is also given in the supplementary material.

Note that the P-risk minimax lower bound is $n^{-2\alpha/(2\alpha+1)}$ since the P-risk convergence rate is slower than or equal to the frequentist minimax rate \citep{lee2018optimal}. Thus, if we set $k$, $a$ and $\epsilon_n$ of the blockwise tapering post-processed posterior such that $k=n^{1/(2\alpha+1)}$,  $p^{1/2}5^{k/2} n\exp(-\lambda n) \preceq\epsilon_n^2\preceq k/n$ and $a>(\alpha+1)/\tau$, then the P-risk convergence rate is the same as the lower bound asymptotically. Thus, the P-risk minimax convergence rate is $n^{-2\alpha/(2\alpha+1)}$, and the blockwise tapering post-processed posterior attains the convergence rate.

\se{Numerical Studies}\label{sec:numerical}
\sse{Simulation}

We compare the blockwise tapering estimator with the tapering estimator using simulation data.
We define the true covariance matrix $\Sigma_0\in\bbR^{p\times p}$ as below. 
Let $\Sigma_0^*=(\sigma_{0,ij}^*)_{1\le i,j\le p}$, where 
\begin{equation*}
	\sigma_{0,ij}^* = \begin{cases}
		1,               & 1\le i=j\le p\\
		\rho |i-j|^{-(\alpha+1)}, &  1\le i\neq j\le p
	\end{cases},
\end{equation*}
and let $\Sigma_0=\Sigma_0^* + \{0.5-\lambda_{\min}(\Sigma_0^*)\}I_p$, which guarantees the minimum eigenvalue of $\Sigma_0$ is bounded away from zero. We set $\rho=0.6$ and $\alpha=0.1$ for $\Sigma_0$ and generate data $Z_1,\ldots,Z_n$ from $N_p(0,\Sigma_0)$ independently, where $p\in\{500,1000\}$ and $n=p/2$. 
Let $ p_0=0.8p$ and fix the positive-definite adjustment parameter $\epsilon_n$ as $0.5$.
We define the error reduction value by choosing the blockwise tapering estimator over the tapering estimator as 
$$d_f (S_n; k,a) = ||\psi\{T_k^{(\epsilon_n)}(S_n)\}-\psi(\Sigma_0)||- || \phi(S_n;2\lfloor ak\log k\rfloor,\epsilon_n)- \psi(\Sigma_0)||.$$      
We repeat generating the simulation data $T$ times, and let $\bbZ_n^{(i)}$ and $S_n^{(i)}$ denote the data and the sample covariance matrix, respectively, in the $i$th repetition for $i\in \{1,2,\ldots, T\}$. 
We summarize the error reduction values from the repetitions as t-value 
\bea
t_f(k,a;T)  = \frac{\sum_{i=1}^T d_f(S_n^{(i)};k,a)/T}{ [\sum_{i=1}^T \{d_f(S_n^{(i)};k,a)-\sum_{i=1}^T d_f(S_n^{(i)};k,a)/T\}^2/T]^{1/2}},
\eea
which is the performance measure for the comparison between the tapering and blockwise tapering estimators.
We also compare the blockwise tapering PPP with the tapering PPP for the same simulation data. 
We define the error reduction value by choosing the blockwise tapering PPP as 
\bea
d_b (\bbZ_n ; k,a) &=& || \hat{C}^{(TPPP)}-\psi(\Sigma_0)||- ||\hat{C}^{(bTPPP)}- \psi(\Sigma_0)||,
\eea
where $ \hat{C}^{(TPPP)}$ and $\hat{C}^{(bTPPP)}$ are the posterior means of the tapering PPP and the blockwise tapering PPP, respectively.
We define the t-value for $T$ repetitions as 
\bea
t_b(k,a;T)  = \frac{\sum_{i=1}^T d_b (\bbZ_n^{(i)} ; k,a)/T}{ [\sum_{i=1}^T \{d_b (\bbZ_n^{(i)} ; k,a)-\sum_{i=1}^T d_b (\bbZ_n^{(i)} ; k,a)/T\}^2/T]^{1/2}}.
\eea

We evaluate $t_f(k,a;100) $ and $t_b(k,a;100)$ for $k\in\{2,3,\ldots,10\}$ and $a\in\{5,10,20\}$. For the post-processed posteriors we generate $1000$ posterior samples in each setting. We represent the result of the evaluations in Figure \ref{figure:block_comparison}.
\begin{figure}[!htbp]
	\centering	\includegraphics[scale=0.9]{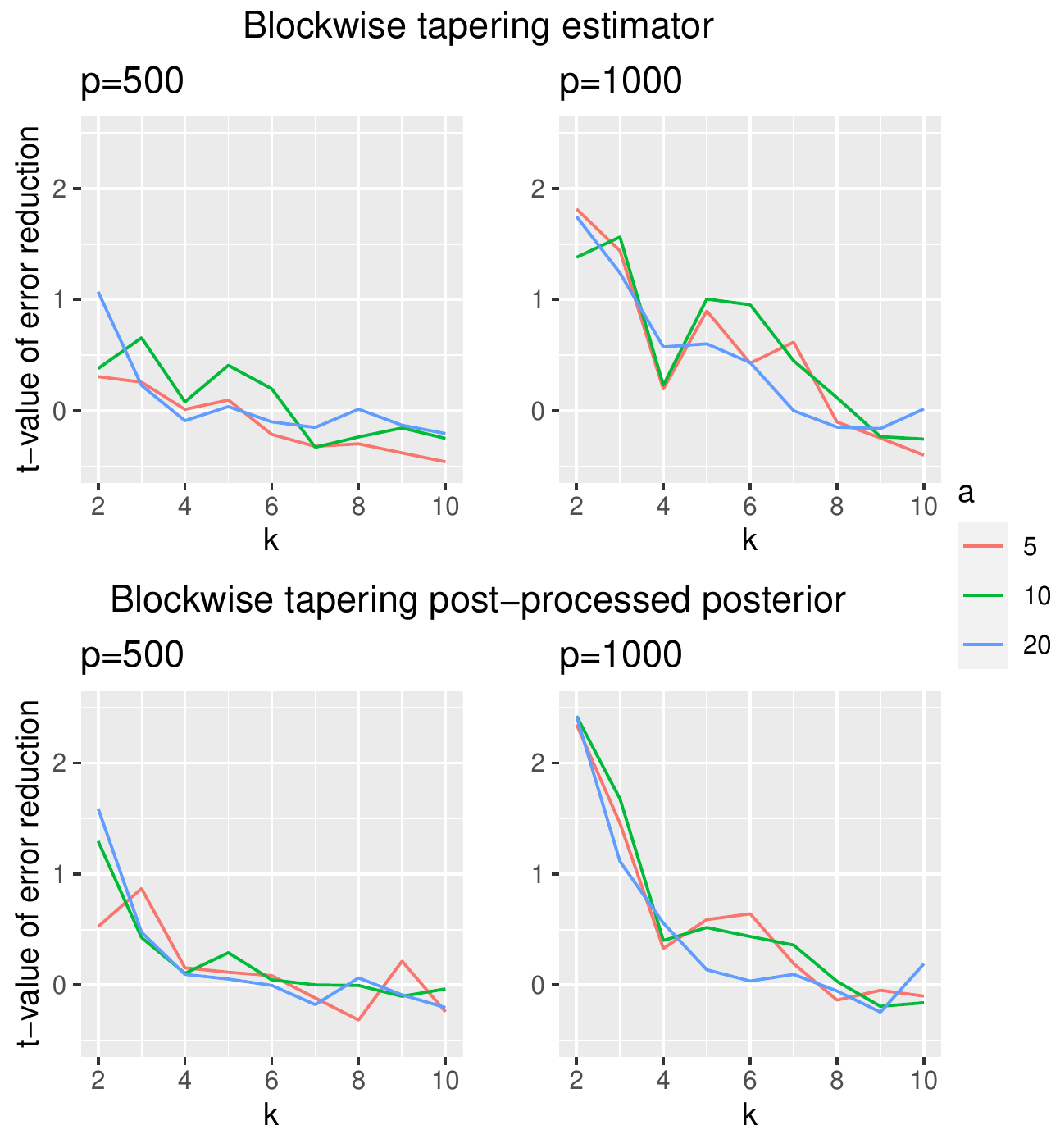}
	\caption{
		The evaluated t-values $t_f(k,a;100) $, the summarized error reductions by choosing the blockwise tapering estimator over the tapering estimator, are represented in the upper plots. 
		The dimension of the covariance $p$ is set to $500$ and $1000$. 
		As the tuning parameters of the methods, $k\in \{2,3,\ldots,10\}$ and $a\in \{5,10,20\}$ are used. 
		The evaluated t-values $t_b(k,a;100) $, the summarized error reductions by choosing the blockwise tapering post-processed posterior over the tapering post-processed posterior, are represented in the lower plots for the same parameters. 
	}
	\label{figure:block_comparison}
\end{figure}
When $p$ is large and $k$ is small, the effects of error reductions by the blockwise tapering estimator and the blockwise tapering post-processed posterior increase.
Note that the convergence rates of the tapering estimator and the tapering PPP contain the additional $\log p/n$ term. The effect of the additional term is increased when another term in the convergence rate, $k/n$, is relatively small. Thus, the error reduction is effective when $p$ is large compared to $k$. The figure also shows that the tapering estimator is slightly better otherwise. 
If $\log p$ is not relatively large, one does not need to abandon the covariates $X_{1:p_0-2\lfloor ak\log k \rfloor}$ by using the blockwise tapering estimator or the blockwise tapering PPP.

Next, we compare the tapering estimator, blockwise tapering estimator, and their Bayesian versions with two other methods: covariance estimation method and multivariate regression method. 
A covariance estimator can be used for the estimation of the conditional mean operator by applying the transformation \eqref{formula:cmoperator}.
We use the banding estimator \citep{bickel2008regularized}, dual maximum likelihood estimator \citep{kauermann1996dualization}, and the banding post-processed posterior \citep{lee2020post} as covariance estimators for comparison.
The multivariate regression method is also used for comparison, since the multivariate linear regression coefficient is the conditional mean operator. We adopt the reduced-rank regression \citep{chen2013reduced}, the sparse reduced-rank regression \citep{chen2012sparse} and the method of sparse orthogonal factor regression (SOFAR) \citep{uematsu2019sofar}.

We need to select tuning parameters for the conditional mean operator estimators.
Based on the tuning parameter selection process, we divide the estimation methods into three categories: frequentist covariance-based method, post-processed posterior method, and multivariate regression method.

The tapering and blockwise tapering estimators belong to the frequentist covariance-based method, and the process of the tuning parameter selection is as follows. 
When a covariance estimator is given, the conditional mean operator and the conditional variance are derived, which yield the conditional distribution under the normality assumption.
The log-likelihood function of the conditional distribution is used for the leave-one-out cross-validation.
Let $\hat{\Sigma}(\bbZ_{n,-i},\tau)$ be a frequentist covariance estimator based on $\bbZ_{n,-i}=(Z_1,\ldots,Z_{i-1},Z_{i+1},\ldots,Z_n)$ given a tuning parameter vector $\tau$. The derived conditional mean operator is $\psi\{\hat{\Sigma}(\bbZ_{n,-i},\tau)\}$, and the conditional variance is 
\bea
\nu\{\hat{\Sigma}(\bbZ_{n,-i},\tau)\}: =\hat{\Sigma}(\bbZ_{n,-i},\tau)_{YY}- \hat{\Sigma}(\bbZ_{n,-i},\tau)_{YX}\{\hat{\Sigma}(\bbZ_{n,-i},\tau)_{XX}\}^{-1}\hat{\Sigma}(\bbZ_{n,-i},\tau)_{XY}   .
\eea
We select $\tau$ as the minimizer of 
\bea
\hat{R}^{(f)}(\tau) = \sum_{i=1}^n\log p[ Y_i \mid \psi\{\hat{\Sigma}(\bbZ_{n,-i},\tau)\} X_{i}, \nu\{\hat{\Sigma}(\bbZ_{n,-i},\tau)\}],
\eea
where $(X_i^T,Y_i^T)^T=Z_i$ and $p(x\mid \mu,\Sigma)$ is the density function of the multivariate normal distribution with mean $\mu$ and covariance $\Sigma$.
Since the conditional variance can not be derived from the blockwise tapering estimator, we use the conditional variance from the tapering estimator in this case.

For the tuning parameter selection of the post-processed posterior methods, we use the Bayesian leave-one-out cross-validation method \citep{gelman2014understanding} to the log-likelihood function of the conditional distribution. Let $\Sigma_1^{(i)},\Sigma_2^{(i)},\ldots,\Sigma_S^{(i)}$ be leave-one-out initial posterior samples which are generated from the initial posterior by $\bbZ_{n,-i}$ for $i\in\{1,2,\ldots,n\}$. We select the tuning parameter vector $\tau$ as the minimizer of 
\bea
\sum_{i=1}^n\log \frac{1}{S}\sum_{s=1}^S p\{Y_i \mid \psi^*(\Sigma_s^{(i)};\tau) X_{i}, \nu^*(\Sigma_s^{(i)};\tau) \},
\eea
where $\psi^*$ and $\nu^*$ are post-processing functions for the conditional mean operator and conditional variance given the tuning parameter $\tau$, respectively.
For the post-processing function of the conditional variance $\nu^*$, the banding PPP uses $\nu\{B_k^{(\epsilon_n)}(\Sigma_s^{(i)})\}$, where $B_k^{(\epsilon_n)}$ is the positive-definite adjusted banding operator defined as
\bea
B_k^{(\epsilon_n)}(\Sigma) = B_k(\Sigma) + ([\epsilon_n -\lambda_{\min}\{B_k(\Sigma)\}]\vee 0)I_p,
\eea
and the tapering and blockwise tapering PPPs use $\nu\{T_k^{(\epsilon_n)}(\Sigma_s^{(i)})\}$.

For the multivariate regression method, we use $10$-fold cross-validation method as \cite{chen2012sparse}, \cite{chen2013reduced} and \cite{uematsu2019sofar} suggested. Note that all the methods contain the rank parameter in the tuning parameters. While we select the rank from $\{0,1,\ldots,10\}$ for the reduced-rank regression, $\{1,2,\ldots,10\}$ is considered for the others. Because if the rank is zero, all the three methods coincide, we only consider the \cite{chen2013reduced}'s method for the zero rank case.

We set $p=200$, $\rho=0.6$ and $\alpha=0.1,0.3$ for $\Sigma_0$ and generate $Z_1,Z_2,\ldots,Z_n$ from $N_p(0,\Sigma_0)$ independently for $n\in\{100,200\}$. 
We repeat generating the simulation data $100$ times for each simulation setting. The performance of each method is measured as 
\bea
\frac{1}{100}\sum_{s=1}^{100} ||\psi(\Sigma_0)- \hat{C}_s||,
\eea
where $\hat{C}_s$ is the point estimator for the conditional mean operator in the $s$th repetition. For the post-processed posterior methods, we use the posterior mean as the point estimator. Table \ref{table:pointerror} gives the simulation error. 
The tapering estimator and the blockwise tapering estimator, and their Bayesian counterparts are the best in all settings. 
The multivariate regression methods, i.e. the reduced-rank regression, sparse reduced-rank regression and sparse orthogonal factor regression, are the worst in all settings. 
Unlike the other covariance-based methods, the bandable or banded covariance structure is not considered in the multivariate regression methods. It appears that the multivariate regression framework does not perform well under the high-dimensional bandable covariance assumption. 

\begin{table}[bt]
	\caption{Spectral norm errors of estimators for the conditional mean operator.}\label{table:pointerror}
		\begin{tabular}{lcccc}
			\\
			&\multicolumn{2}{c}{ $n=100$}&\multicolumn{2}{c}{ $n=200$} \\
			& $\alpha=0.1$ & $\alpha=0.3$ & $\alpha=0.1$ & $\alpha=0.3$ \\[5pt]
			Tapering estimator & 0.255  & 0.240 &  0.206  &  0.188  \\
			Blockwise tapering estimator & 0.255  & 0.240 &  0.206  &  0.188  \\
			Banding estimator  & 0.319  & 0.293 &  0.290 &  0.247  \\
			Dual maximum likelihood estimator & 0.365 &   0.357 &  0.319  & 0.273  \\
			Tapering post-processed posterior  & 0.257  &  0.246 &  0.208  & 0.193  \\
			Blockwise tapering post-processed posterior  & 0.257  &  0.246 &  0.208  & 0.193  \\
			Banding post-processed posterior  & 0.323 &   0.323 &  0.303 &  0.256 \\
			Reduced-rank regression  & 0.509 &  0.488 &  0.509  & 0.488\\
			Sparse reduced-rank regression  & 3.324 &  3.309  & 2.818 &  2.902\\
			Sparse orthogonal factor regression  & 1.823 &  1.819  & 1.679 &  1.703
		\end{tabular}
\end{table}

\sse{Application to forecasting traffic speed} 

We apply the proposed methods to multivariate regression analysis for small area spatio-temporal data, and use this application to forecast traffic speed in Yeoui-daero, a road in Seoul.

Suppose spatio-temporal data are observed in $S$ spatial regions and $T$ times, where $S$ and $T$ are positive integers.
Let $X_{s,t}$ be a random variable at $s$th spatial index and $t$th time index, $s=1,\ldots,S$ and $t=1,\ldots,T$. 
We assume 
\bean
&&E[\{X_{s_1,t_1}-E(X_{s_1,t_1})\}\{ X_{s_2,t_2}-E( X_{s_2,t_2})\}] \nonumber\\
&\le& r(|t_1-t_2|), ~ s_1,s_2 \in \{1,\ldots,S\},t_1,t_2 \in \{1,\ldots,T\},\label{formula:STassumption}
\eean
where $r$ is a real-valued function from the non-negative integer space, and is assumed to be a decreasing function.
Rearranging $(X_{s,  t})_{s=1,\ldots,S,t=1,\ldots,T} $, we define $Z\in \bbR^{TS}$ as 
\bean\label{formula:rearrange}
Z = ( X_{1,1},X_{2,1},\ldots,X_{S,1}, X_{1,2},X_{2,2},\ldots,X_{S,T} ) ,
\eean
and let $E(ZZ^T)=\Sigma_0=(\sigma_{0,ij})$.
We show that $\Sigma_0$ is a bandable covariance, if the decreasing rate of $r(x)$ is $x^{-\alpha-1}$.
Note that if $mS \le |i-j| < (m+1)S$ for $m\in \{0,1,\ldots,T-1\}$ and $i,j\in\{1,2,\ldots,TS\}$, then the time index difference between $Z_i$ and $Z_j$ is at least $m$. This observation and assumption \eqref{formula:STassumption} give 
$|\sigma_{0,ij}|\le r(\lfloor|i-j|/S \rfloor )$
and
\bea
\sup_j \sum_i \{ |\sigma_{0,ij}| : |i-j| \ge  k \} &\le&
\sup_j \sum_i \{ |\sigma_{0,ij}| : |i-j| \ge \lfloor k/S \rfloor S \} \\
&\le& \sup_j \sum_{m= \lfloor k/S \rfloor S}^{T-1} \sum_{i\in I_{m}^{(j)}} |\sigma_{0,ij}|\\
&\le& S\sum_{m= \lfloor k/S \rfloor S}^{T-1} r(m),
\eea
where $I_m^{(j)} = \{ i\in \{1,2,\ldots,p\} :  \lfloor |i-j|/S \rfloor = m \}$.
If the decreasing rate of $r(x)$ is $x^{-\alpha-1}$, then 
\bea
\sup_j \sum_i \{ |\sigma_{0,ij}| : |i-j| \ge  k \} \le C k^{-\alpha},
\eea
for some positive constant $C$.
Thus, $\Sigma_0$ is a bandable covariance, and the proposed methods for the conditional mean operator under the bandable covariance assumption can be used to predict $X_{1:S,t_0+1:T}$ given $X_{1:S,1:t_0}$.

Based on the rearrangement \eqref{formula:rearrange} and the proposed methods for the conditional mean operator under the bandable covariance assumption, we forecast traffic speed in Yeoui-daero using data from TOPIS \citep{topis}.
In the traffic speed data in Yeoui-daero, a daily data set consists of observations in 8 spatial indexes and 24-time indexes. Let a daily traffic speed be $(X_{s,t})_{1\le s\le 8,1\le t\le 24}$, where
time index $t$ indicates the time interval from $(t-1)$ o’clock to $t$ o’clock, and the allocation of the spatial index $s$ is given in Figure \ref{fig:yeoui-daero}.
We rearrange $(X_{s,t})_{1\le s\le 8,1\le t\le 24}$ as \eqref{formula:rearrange}, and apply the proposed estimators to forecast $X_{1:8,18:24}$ given $X_{1:8,1:17}$.
\begin{figure}[!htbp]
	\centering
	\includegraphics[height=12cm,width=12cm]{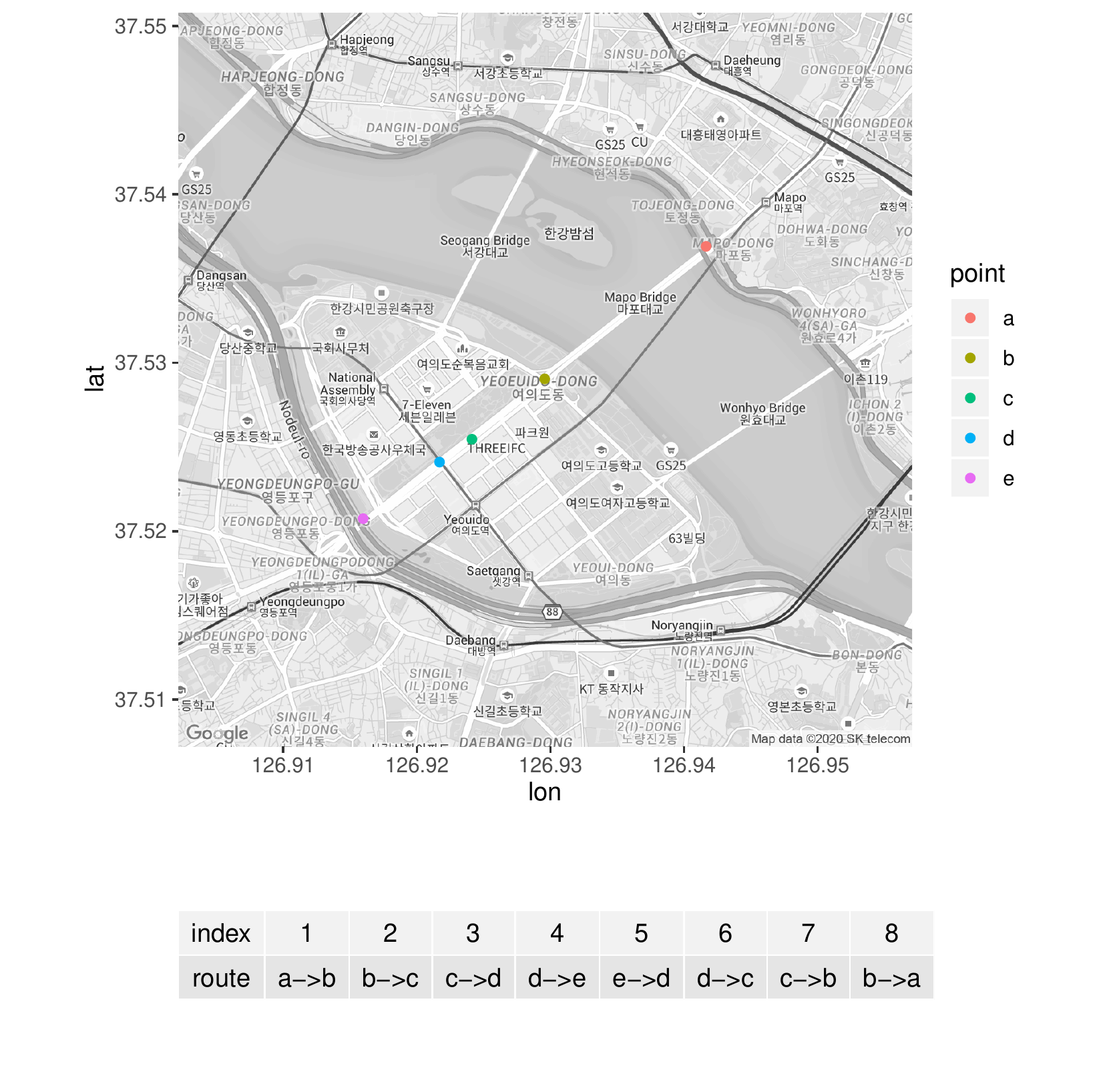}
	\caption{ The eight routes in Yeoui-daero and their index allocation. }
	\label{fig:yeoui-daero}
\end{figure}

In the data from TOPIS, we use data from January to October in $2020$, excluding weekend data sets and missing data sets. We have $172$ days observations which are donoted by $Z_1,Z_2,\ldots,Z_{172} \in \bbR^{192}$. To apply the proposed methods, we use mean-centered observations $\tilde{Z}_1,\tilde{Z}_2,\ldots,\tilde{Z}_{172} \in \bbR^{192}$.  
For the performance measure, 
let the training data be $\bbZ_{train} = (\tilde{Z}_1,\tilde{Z}_2,\ldots,\tilde{Z}_{86} )$, and the test data be $\bbZ_{test} = (\tilde{Z}_{87},\tilde{Z}_{88},\ldots,\tilde{Z}_{172} )$. 
The forecast errors are summarized as 
\bea
\frac{1}{86}\sum_{i=87}^{172}||\hat{C}(\bbZ_{train})(\tilde{Z}_i)_{1:p_0} - (\tilde{Z}_i)_{(p_0+1):p}||_2,
\eea
where $p_0=17\times 8$ and $\hat{C}(\bbZ_{train})$ is a point estimator for the conditional mean operator based on $\bbZ_{train}$.
The summarized forecast errors are represented in Table \ref{table:trafficerror}, which shows the tapering and blockwise tapering estimators and their Bayesian versions are the best among all methods.

\begin{table}[bt]
	\caption{The root mean square errors of forecast results for traffic speed in Yeoui-daero.} \label{table:trafficerror}
		\begin{tabular}{|l|c|}
			\hline
			Method                   &  Error \\ \hline
			Tapering estimator & 2.80 \\ \hline
			Blockwise tapering estimator & 2.80   \\ \hline
			Banding estimator  & 2.85   \\ \hline
			Dual maximum likelihood estimator & 3.33  \\ \hline
			Tapering post-processed posterior  & 2.78 \\ \hline
			Blockwise tapering post-processed posterior  & 2.78 \\ \hline
			Banding post-processed posterior  & 2.87 \\ \hline
			Reduced-rank regression  & 3.59 \\ \hline
			Sparse reduced-rank regression  & 3.39\\ \hline
			Sparse orthogonal factor regression  & 4.31\\ \hline
		\end{tabular}
\end{table}


\se{Discussion}

We have considered the estimation of the conditional mean operator under the bandable covariance assumption, which is useful for the multivariate linear regression when there is a natural order in the variables. 
We showed that the plug-in estimator by the tapering estimator of covariance, which is the minimax optimal estimator for the class of bandable covariance, is sub-optimal for the conditional mean operator. This observation implies that 
when a function of the covariance matrix is to be estimated, 
the plug-in estimator by a minimax optimal covariance estimator may not be optimal.
We have proposed the blockwise tapering estimator and the blockwise tapering post-processed posterior as minimax-optimal estimators for the conditional mean operator under the bandable covariance assumption.
We constructed the estimators by modifying the tapering estimator and the tapering post-processed posterior to exclude the covariates which have small partial correlations with the response variables. 
Using the numerical studies, we also showed that the blockwise tapering estimator and the blockwise tapering post-processed posterior have smaller errors when $p$ is large enough.


\appendix

\se{Proofs of main theorems}\label{sec:proof}
In this section, we prove Theorems \ref{theorem:frequpper} and \ref{theorem:condmeanlower}, which give the convergence rate of the blockwise tapering estimator and the lower bound of the minimax risk, respectively. The proofs of the other theorems and lemma in Sections \ref{sec:CM_bandable} and \ref{sec:btppp} are given in the supplementary material.

We give notations for the proofs. 
Let $||\Sigma||_F =tr(\Sigma \Sigma^T)$ and $||\Sigma||_r$ be the Frobenius norm and the matrix $r$-norm for a covariance matrix $\Sigma$, respectively.
We also let $W_p(B_0,\nu_0)$ be Wishart distribution of which density function is 
\bea
\pi(\Sigma) \propto |\Sigma|^{(\nu_0-p-1)/2}e^{-tr(B_0^{-1}\Sigma)/2}, ~ \Sigma\in\calC_p,
\eea
where $B_0\in \calC_p$ and $\nu_0>p-1$.

\sse{ Proof of Theorem 4}\label{ssec:thm4}

In this subsection, we show the convergence rate of the blockwise tapering estimator by proving Theorem \ref{theorem:frequpper}.
First, we present Lemmas \ref{lemma:l1bandinv}-\ref{lemma:upperblockinv} necessary for the proof of Theorem \ref{theorem:frequpper}.

\begin{lemma}\label{lemma:l1bandinv}
	Let $p$ and $k$ be positive integers with $k<p$ and suppose $\Sigma_0 \in \calF_{p,\alpha}(M,M_0,M_1)$. There exist some positive constants $C_1$ and $C_2$ depending only on $M$, $M_0$ and $M_1$ such that 
	\bea
	||\Sigma_0^{-1} - B_k(\Sigma_0)^{-1} ||_1 &\le& C_1 k^{-\alpha}\\
	||\Sigma_0^{-1} - T_k(\Sigma_0)^{-1} ||_1 &\le& C_2 (\lfloor k/2\rfloor )^{-\alpha},
	\eea
	for all sufficiently large $k$.
\end{lemma}
The proof of this lemma is given in the supplementary material.

\begin{lemma}\label{lemma:evminfreq}
	Let $n$, $k$ and $p$ be positive integers with $k \le  p$ and
	suppose $\Sigma_0\in\calF_{p,\alpha}(M,M_0,M_1)$. If $c\le \lambda_{\min}(\Sigma_0)/2$ and $\lfloor k/2\rfloor> \{4M/\lambda_{\min}(\Sigma_0)\}^{1/\alpha}$, then
	\bea
	P_{\Sigma_{0}} [ \lambda_{\min} \{T_k(S_n)\}   \le c] \le 2 p 5^k \exp(-\lambda n) ,
	\eea
	for some positive constant $\lambda$ depending only on $M_0$ and $M_1$.
\end{lemma}
The proof of this lemma is given in the supplementary material.

\begin{lemma}\label{lemma:upper1}
	Let $n$, $p$ and $k$ be positive integers with $k\le p$ and 
	suppose $\Sigma_0\in\calF_{p,\alpha}(M,M_0,M_1)$. If $ k\vee \log p =o(n)$ and $\epsilon_n=O(1)$, then there exists some positive constant $C$ depending only on $M_0$ and $M_1$ such that
	\bea
	E_{\Sigma_0}(||T_k^{(\epsilon_n)}(S_n)-\Sigma_0||^4)\le C,
	\eea
	for all sufficiently large $n$.

\end{lemma}
The proof of this lemma is given in the supplementary material.

\begin{lemma}\label{lemma:upperblockinv}
	Suppose $\Sigma_0\in\calF_{p,\alpha}(M,M_0,M_1)$. 
	Let $q$ and $k$ be positive integers with $k<q\le p$, and let $(\Sigma_0)_{1:q,1:q}$ and $ \{T_k^{(\epsilon_n)}(S_n)\}_{1:q,1:q}$ denoted by $\Sigma_{0,11}$ and $ T_k^{(\epsilon_n)}(S_n)_{11}$, respectively.
	If $k\vee \log p =o(n)$, $\lfloor k/2\rfloor> \{4M/\lambda_{\min}(\Sigma_0)\}^{1/\alpha}$ and $\epsilon_n=O(1)$, then there exist some positive constants $C$ and $\lambda$ depending only on $M$, $M_0$, $M_1$ and $\alpha$ such that
	\bea
	E_{\Sigma_0} (|| T_k^{(\epsilon_n)}(S_n)_{11}^{-1} -\Sigma_{0,11}^{-1} ||^2) \le C\Big\{  \frac{k + \log p}{n} + k^{-2\alpha} + \epsilon_n^2 +  \frac{p^{1/2} 5^{k/2} \exp(-\lambda n) }{\epsilon_n^2}  \Big\},
	\eea
	for all sufficiently large $n$.
\end{lemma}
The proof of this lemma is given in the supplementary material.

Now we prove Theorem \ref{theorem:frequpper}.
\begin{proof}[Proof of Theorem~\ref{theorem:frequpper}]
	We have
	\bean
	&&E(|| \psi(\Sigma_0)  - T_k(S_n)_{YX} \Lambda^{(\epsilon_n)}\{ T_k(S_n)_{XX}; 2\lfloor ak\log k \rfloor\}||^2) \nonumber\\
	&\le & 2||\psi(\Sigma_0) - T_k(\Sigma_0)_{YX}\Lambda^{(0)}\{T_k(\Sigma_{0,XX});2\lfloor ak\log k \rfloor\}||^2 \label{formula:frequpper1}\\
	&&+ 4E(||T_k(\Sigma_{0})_{YX}- T_k(S_n)_{YX}  ||^2 ~ ||\Lambda^{(\epsilon_n)}\{ T_k(S_n)_{XX};2\lfloor ak\log k \rfloor\} ||^2)\label{formula:frequpper2}\\
	&&+ 4||T_k(\Sigma_0)||^2E( ||\Lambda^{(0)}\{T_k(\Sigma_{0,XX});2\lfloor ak\log k \rfloor\}-\Lambda^{(\epsilon_n)}\{ T_k(S_n)_{XX};2\lfloor ak\log k \rfloor\}||^2).\label{formula:frequpper3}
	\eean
	By Lemma \ref{lemma:trueerror}, there exist some positive constants $C_1$ and $\lambda_1$ depending only on $M$, $M_0$, $M_1$ and $\alpha$ such that
	\bea
	\eqref{formula:frequpper1}\le     C_1k^{-2(\alpha \wedge (a\lambda_1 -1))},
	\eea
	for all sufficientlay large $k$ with $\lfloor ak\log k \rfloor/2 \ge k$ and $p_0> 2\lfloor ak\log k \rfloor$.
	
	Next, we show the upper bound of \eqref{formula:frequpper2}. Let $\tilde{M}:=M_{p_0-2\lfloor ak\log k\rfloor+1}^{(2\lfloor ak\log k\rfloor)} $ in this proof. 
	By the definition of $\Lambda^{(\epsilon_n)}$, we have
	\bea
	||\Lambda^{(\epsilon_n)}\{ T_k(S_n)_{XX};2\lfloor ak\log k \rfloor\} || = ||T_k^{(\epsilon_n)}[\tilde{M}\{(S_n)_{XX}\}]^{-1}||,
	\eea
	and
	\bea
	\eqref{formula:frequpper2}
	&= &4E(||T_k(\Sigma_{0})_{YX}- T_k(S_n)_{YX}||^2 ||T_k^{(\epsilon_n)}[\tilde{M}\{(S_n)_{XX}\}]^{-1}||^2)\\
	&\le& \frac{4}{\epsilon_n^2}E\{||T_k(\Sigma_{0})_{YX}- T_k(S_n)_{YX} ||^2 I(\lambda_{\min}[\tilde{M}\{(S_n)_{XX}\}]<M_1/2)\}\\
	&&+ \frac{16}{M_1^2} E(||T_k(\Sigma_{0})_{YX}- T_k(S_n)_{YX} ||^2) \\
	&\le& \frac{4}{\epsilon_n^2} E(||T_k(\Sigma_{0})- T_k(S_n) ||^4)^{1/2} P\{\lambda_{\min}(S_n)<M_1/2\}^{1/2}\\
	&&+ \frac{16}{M_1^2} E(||T_k(\Sigma_{0})- T_k(S_n) ||^2) \\
	&\le& \frac{4}{\epsilon_n^2} E(2^3||\Sigma_{0}- T_k(S_n) ||^4 + 2^3 M^4(\lfloor k/2\rfloor)^{-4\alpha})^{1/2} P\{\lambda_{\min}(S_n)<M_1/2\}^{1/2}\\
	&&+ \frac{16}{M_1^2} E(2||\Sigma_{0}- T_k(S_n) ||^2 + 2M^2(\lfloor k/2\rfloor )^{-2\alpha}) \\
	&\le&  C_2\Big\{\frac{1}{\epsilon_n^2} p^{1/2}5^{k/2} \exp(-\lambda_2n) + \frac{k+\log p}{n} + k^{-2\alpha}\Big\},
	\eea
	for some positive constants $C_2$ and $\lambda_2$ depending only on $M$, $M_0$, $M_1$ and $\alpha$. The first inequality holds since $$\lambda_{\min}(T_k^{(\epsilon_n)}[\tilde{M}\{(S_n)_{XX}\}])\ge \epsilon_n.$$
	The third inequality holds since
	\bean
	||T_k(\Sigma_0)-\Sigma_0||_1 &\le& ||B_{\lfloor k/2\rfloor }(\Sigma_0)-\Sigma_0 ||_1\nonumber\\
	&\le& M(\lfloor k/2\rfloor )^{-\alpha}.
	\eean
	The last inequality holds by Lemmas \ref{lemma:evminfreq}, \ref{lemma:upper1} and Theorem 2 of \cite{cai2010optimal}.
	For the upper bound of \eqref{formula:frequpper3}, we have that there exists some positive constant $C_3$ depending only on $M$, $M_0$ and $M_1$ such that
	\bea
	&&E( ||\Lambda^{(0)}\{T_k(\Sigma_{0,XX});2\lfloor ak\log k \rfloor\}-\Lambda^{(\epsilon_n)}\{ T_k(S_n)_{XX};2\lfloor ak\log k \rfloor\}||^2)\\
	&=& E(||T_k\{\tilde{M}(\Sigma_{0,XX} )\}^{-1}-   T_k^{(\epsilon_n)}[\tilde{M}\{(S_n)_{XX}\}]^{-1}  ||^2)\\
	&\le& 2E(||\tilde{M}(\Sigma_{0,XX} )^{-1} - T_k^{(\epsilon_n)}[\tilde{M}\{(S_n)_{XX}\}]^{-1}||^2)\\
	&&+2 ||T_k\{\tilde{M}(\Sigma_{0,XX} )\}^{-1} -\tilde{M}(\Sigma_{0,XX} )^{-1}||^2\\
	&\le& 2E(||\tilde{M}(\Sigma_{0,XX} )^{-1} - T_k^{(\epsilon_n)}[\tilde{M}\{(S_n)_{XX}\}]^{-1}||^2)+ C_3(\lfloor k/2\rfloor)^{-2\alpha},
	\eea     
	for all sufficiently large $k$.
	The first equality is satisfied by the definition of $\Lambda^{(0)}$ and $\Lambda^{(\epsilon_n)}$. The last inequality holds by Lemma \ref{lemma:l1bandinv} since $\tilde{M}(\Sigma_{0,XX} ) \in \calF_{2\lfloor ak\log k\rfloor,\alpha}(M,M_0,M_1)$. We apply Lemma \ref{lemma:upperblockinv} and obtain that there exist some positive constants $C_4$ and $\lambda_3$ depending only on $M$, $M_0$, $M_1$ and $\alpha$ such that
	\bea
	&&E(||\tilde{M}(\Sigma_{0,XX} )^{-1} - T_k^{(\epsilon_n)}[\tilde{M}\{(S_n)_{XX}\}]^{-1}||^2) \\
	&\le& C_4 \Big\{\frac{k+\log(2ak\log k)}{n} + \epsilon_n^2 +  k^{-2\alpha} + \frac{(2ak\log k)^{1/2}5^{k/2}\exp(-\lambda_1 n)}{\epsilon_n^2}\Big\},
	\eea
	for all sufficiently large $n$.
	Thus, there exists some positive constant $C_5$ depending only on $M$, $M_0$, $M_1$ and $\alpha$ such that
	\bea
	\eqref{formula:frequpper3}\le C_5 \Big\{\frac{k+\log(2ak\log k)}{n} + \epsilon_n^2 +  k^{-2\alpha} + \frac{(2ak\log k)5^k\exp(-\lambda_1 n)}{\epsilon_n^2}\Big\},
	\eea
	for all sufficiently large $n$ and $k$. Collecting the upper bounds of \eqref{formula:frequpper1}, \eqref{formula:frequpper2} and \eqref{formula:frequpper3}, we complete the proof.	
	
\end{proof}

\sse{Proof of Theorem 5}\label{ssec:thm5}

In this subsection, we prove Theorem \ref{theorem:condmeanlower} which gives the lower bound of the minimax risk for the conditional mean operator under the bandable covariance assumption.
For probability measures $P_{\theta}$ and $P_{\theta'}$, we define 
\bea
|| P_\theta - P_{\theta'} ||_1 &=& \int | p_{\theta'}-p_{\theta} | d\nu\\
|| P_\theta \wedge P_{\theta'} || &=& \int  p_{\theta'}\wedge p_{\theta}  d\nu,
\eea
where $p_{\theta}$ and $p_{\theta'}$ are probability density functions of $P_\theta$ and $P_{\theta'}$, respectively, with respect to the reference measure $\nu$. 
First, we provide Lemma \ref{lemma:ineqTV} which is a reformulation of the proof of Lemma 6 in \cite{cai2010optimal}.
\begin{lemma}\label{lemma:ineqTV}
	Suppose $\Sigma$ and $\Sigma'$ are $p\times p$-positive definite matrices. 
	Let $P_\Sigma$ be the joint distribution of $X_1,X_2,\ldots,X_n$, which are independent and identically generated from $N_p(0,\Sigma)$.
	If $||\Sigma-\Sigma'||_2  ( ||\Sigma^{-1}||_2 \wedge ||\Sigma'^{-1}||_2   )  < 1/2$, then 
	\bea
	|| P_\Sigma - P_{\Sigma'} ||_1^2 \le n(||\Sigma^{-1}||_2 \wedge ||\Sigma'^{-1}||_2)^2 ||\Sigma-\Sigma'||_F^2.
	\eea
\end{lemma}
The proof of this lemma is given in the supplementary material.

\begin{proof}[Proof of Theorem \ref{theorem:condmeanlower}]
	Let $\calC_{p,p_0} = \{ A \in \bbR^{p\times p} : A_{1:p_0,1:p_0} \in \calC_{p_0} \}$, and $(\calC_{p,p_0})^{ \chi} $ be the space of estimators on $\calC_{p,p_0}$, where $\chi$ is the sample space.
	Since for an arbitrary $B \in \bbR^{(p-p_0)\times p_0}$ there exists $A_B\in \calC_{p,p_0}$ such that $\psi(A_B;p_0)=B$, it suffices to show 
	\bea
	\inf_{\hat{A} \in (\calC_{p,p_0})^{ \chi} }\sup_{ \Sigma_0\in \calF_{p,\alpha} }E (|| \psi(\Sigma_0;p_0) -\psi(\hat A;p_0) ||^2 )
	&\ge& C  n^{-2\alpha/(2\alpha+1)},
	\eea
	for some positive constant $C$.
	Let $d(\hat{\Sigma},\Sigma_0) = || \hat{\Sigma}_{YX}\hat{\Sigma}_{XX}^{-1} -\Sigma_{0,YX}\Sigma_{0,XX}^{-1} ||_2$, which is a semimetric on $\calC_{p,p_0}$.
	For a positive integer $k$ with $k<p_0/2$, let $\Theta =  \{0,1\}^k$, and let $H(\theta,\theta')$ be the Hamming distance on $\Theta$. 
	We define 
	\bea
	\Sigma(\theta) = \bar{M} I_p +  \tau \sum_{m=1}^k \theta_m  B(m+p_0;k) , \text{ }\theta=(\theta_1,\theta_2,\ldots,\theta_k)\in \Theta ,
	\eea
	where $\tau =\{M (2k)^{-\alpha-1}\}\wedge 
	\{(M_0-\bar{M})/(2k)\}$, $\bar{M}=(M_0+M_1)/2$, $B(l;k)= (b_{ij})_{1\le i,j\le p}$ and  
	\bea
	b_{ij} = I(i=l \text{ and } l-2k\le j\le l-1 \text{, or } j=l \text{ and } l-2k\le i\le l-1) .
	\eea
	Note that $\Sigma(\theta) \in \calF_{p,\alpha}(M,M_0,M_1)$ for all $\theta\in \Theta$. 
	Thus, 
	\bean
	\inf_{\hat{A} \in (\calC_{p,p_0})^{ \chi} }\sup_{\Sigma_0\in  \calF_{p,\alpha} }E || \psi(\Sigma_0;p_0) -\psi(\hat A;p_0) ||^2 
	&\ge& 	\inf_{\hat{A} \in (\calC_{p,p_0})^{ \chi} }\sup_{ \theta \in \Theta } E || \psi(\Sigma(\theta);p_0) -\psi(\hat A;p_0) ||^2\nonumber \\
	&=&	\inf_{\hat{A} \in (\calC_{p,p_0})^{ \chi} }\sup_{ \theta \in \Theta } E[d^2\{\Sigma(\theta) ,\hat A\}] \label{formula:lower2}
	\eean
	By the Assouad lemma \citep[Lemma 4]{cai2010optimal}, we have, for all $s>0$,
	\bean
	\sup_{\theta\in\Theta} 2^s E[ d^s\{\hat{A},\Sigma(\theta)\}] \ge \min_{H(\theta,\theta')\ge 1} \frac{d^s\{\Sigma(\theta),\Sigma(\theta')\}}{H(\theta,\theta')} \frac{k}{2} \min_{H(\theta,\theta')=1} || P_\theta \wedge P_{\theta'}||,\label{thm:lower1}
	\eean
	where $P_{\theta}$ is the joint distribution of $n$ independent observations from the multivariate normal distribution with mean zero and covariance $\Sigma(\theta)$.
	
	First, we show the lower bound of 
	\bea
	\frac{d^s\{\Sigma(\theta),\Sigma(\theta')\}}{H(\theta,\theta')}.
	\eea
	For vector $v=\{ I (p_0-k< i \le p_0)\}_{1\le i\le p_0}$ and all $\theta,\theta' \in\Theta$ with $\theta \neq \theta'$, we have
	\bea
	d^2\{\Sigma(\theta),\Sigma(\theta')\} &\ge&
	\frac{||\{\Sigma(\theta)_{YX}- \Sigma(\theta')_{YX}\} (\bar{M} I_{p_0})^{-1}v||^2}{||v||^2} \\
	&\ge& 
	\bar{M}^{-2} \frac{||\{\Sigma(\theta)_{YX}- \Sigma(\theta')_{YX}\}v||^2}{||v||^2}\\
	&\ge& \bar{M}^{-2}H(\theta,\theta')\tau^2 k ,
	\eea
	which implies 
	\bean
	\min_{H(\theta,\theta')\ge1}\frac{d^2\{\Sigma(\theta),\Sigma(\theta')\}}{H(\theta,\theta')}\ge 
	\bar{M}^{-2}\tau^2 k  .\label{thm:lower2}
	\eean
	
	Next, we consider the lower bound of $	\min_{H(\theta,\theta')=1} ||P_\theta \wedge P_{\theta'}||$.
	We assume $H(\theta,\theta')=1$, $2\tau k <\bar{M}$, $2\tau k/(\bar{M}-2\tau k)<1/2$ and $(\sqrt{2}n^{1/2}k^{1/2}\tau)/(\bar{M}-2\tau k) \le 1/2$.
	We have 
	\bea
	||\Sigma(\theta)-\Sigma(\theta')||_2&\le& ||\Sigma(\theta)-\Sigma(\theta')||_1\\
	& \le&  2\tau k,
	\eea
	and 
	\bea
	||\Sigma(\theta)^{-1}|| &=& \lambda_{\min}\{\Sigma(\theta)\}^{-1} \\
	&\le& \frac{1}{ \bar{M} - ||\Sigma(\theta)- \bar{M} I_p||}\\
	&\le&  \frac{1}{ \bar{M} - 2\tau k},
	\eea
	where the first inequality holds by Lemma 4.14 in \cite{lee2020post}.
	Since $2\tau k/(\bar{M}-2\tau k)<1/2$, Lemma \ref{lemma:ineqTV} gives
	\bean
	\min_{H(\theta,\theta')=1} ||P_\theta \wedge P_{\theta'}|| &=&  1- \max_{H(\theta,\theta')=1} ||P_\theta - P_{\theta'}||_1/2 \nonumber\\
	&\ge& 1 -   \frac{n^{1/2}|| \Sigma(\theta) -\Sigma(\theta') ||_F}{2(\bar{M}-2\tau k)} \nonumber\\
	&\ge& 1- \frac{\sqrt{2}n^{1/2}k^{1/2}\tau}{\bar{M}-2\tau k}, \label{formula:condmeanlower}
	\eean
	where the last inequality holds since $|| \Sigma(\theta) -\Sigma(\theta') ||_F\le (8k\tau^2)^{1/2}$.
	Since $(\sqrt{2}n^{1/2}k^{1/2}\tau)/(\bar{M}-2\tau k) \le 1/2$,
	\bea
	\min_{H(\theta,\theta')=1} ||P_\theta \wedge P_{\theta'}|| \ge 1/2.
	\eea
	Thus, collecting inequalities \eqref{formula:lower2}, \eqref{thm:lower1} and \eqref{thm:lower2}, we get
	\bea
	\inf_{\hat{\Sigma}}\max_{\theta\in\{0,1\}^k} 2^2 \bbE_\theta d^2\{\hat{\Sigma},\Sigma(\theta)\}
	&\ge& ck^2\tau^2,
	\eea
	for some positive constant $c$. Since $\tau \le M(2k)^{-\alpha -1}$
	we obtain the desired minimax lower bound by setting $k=(\gamma n)^{1/(2\alpha+1)}/2$, where $\gamma= 16M^2 / (\bar{M}^2)$.
	
	Finally, we check the assumed conditions on $k$:
	\bea
	2k &<& p_0\\
	2\tau k &<& \bar{M}\\
	2\tau k / (\bar{M}-2\tau k) &<& 1/2\\
	\frac{\sqrt{2}n^{1/2}k^{1/2}\tau}{\bar{M}-2\tau k} &\le& 1/2.
	\eea
	Note $\tau \le M(2k)^{-\alpha -1}$ and $k =(\gamma n)^{1/(2\alpha+1)}/2$. 
	The first condition is satisfied when $(\gamma n)^{1/(2\alpha+1)}  < p_0$.
	For the other conditions, we have
	\bea
	2\tau k &\le& M  (\gamma n)^{-\alpha/(2\alpha+1)}\\
	2\tau k / (\bar{M}-2\tau k) &\le&\frac{M  (\gamma n)^{-\alpha/(2\alpha+1)}}{\bar{M}-	M  (\gamma n)^{-\alpha/(2\alpha+1)}} \\
	\frac{\sqrt{2}n^{1/2}k^{1/2}\tau}{\bar{M}-2\tau k} &\le& \frac{M/\gamma^{1/2}}{\bar{M}-M(\gamma n)^{-\alpha/(2\alpha+1)}} .
	\eea
	The first and second upper bounds of these inequalities can be arbitrary small numbers for all sufficiently large $n$. The last upper bound is smaller than $1/2$ for all sufficiently large $n$ since $\gamma= 16M^2 / (\bar{M}^2)$. Thus, the assumed conditions on $k$ are satisfied for all sufficiently large $n$.


\end{proof}

\bibliographystyle{dcu}
\bibliography{cov-ppp}

\end{document}


\maketitle

	In this supplementary material, we give the proofs of Theorems \ref{main-theorem:bandableinv}-\ref{main-theorem:taperupper}, Lemma \ref{main-lemma:trueerror} and Theorems \ref{main-theorem:TPPPupper}-\ref{main-theorem:PPPupper} in the main paper, and the proofs of lemmas given in Section \ref{main-sec:proof}.

	\section{Proof of Theorem 1}

	In this section, we prove Theorem \ref{main-theorem:bandableinv}. 
	First, we present Lemma \ref{lemma:bandinv} which is necessary for the proof of Theorem \ref{main-theorem:bandableinv}. 
	\begin{lemma}\label{lemma:bandinv}
		Let $p$ and $k$ be positive integers with $k<p$. 
		Suppose $\Sigma\in\calC_p$ is a $k$-band matrix, and let $\Sigma^{-1}=(w_{ij})$. For all $a>0$ and all sufficiently large $k$ with $p>k \vee (ak\log k)$,
		\bea
		\max_j\sum_i \{ |w_{ij}|: |i-j|>ak\log k\} \le -\frac{2C}{\log q} k^{2a\log q +1},
		\eea
		where $q=(\kappa^{1/2}-1)/(\kappa^{1/2}+1)$, $C= \{||\Sigma^{-1}||\vee (1+\kappa^{1/2})^2\}/(2||\Sigma||)$, and $\kappa$ is the spectral condition number of $\Sigma$.
	\end{lemma}
	\begin{proof}
		Since $\Sigma$ is a $k$-band matrix, 
		\bea
		|w_{ij}|\le C q^{2|i-j|/k}, ~i,j \in \{1,2,\ldots,p\},
		\eea
		where $q=(\kappa^{1/2}-1)/(\kappa^{1/2}+1)$, $C= \{||\Sigma^{-1}||\vee (1+\kappa^{1/2})^2\}/(2||\Sigma||)$ and $\kappa$ is the spectral condition number of $\Sigma$ defined as $\lambda_{\max}(\Sigma)/\lambda_{\min}(\Sigma)$ \citep[Theorem 2.4]{demko1984decay}. 
		Note that 
		\bea
		\max_j\sum_{\{i:|i-j|>ak\log k \}} q^{2|i-j|/k} &\le&  \frac{2}{1-q^{2/k}} q^{2a\log k }\\
		&=&  \frac{2}{k(1-q^{2/k})} k^{\log q^{2a} +1},
		\eea
		and 
		\bea
		\lim\limits_{k\lra\infty}k(q^{2/k}-1) = \Big(\frac{d q^{2x}}{dx}\Big)_{x=0} = 2\log q.
		\eea
		Thus, we get
		\bea
		\max_j\sum_{\{i:|i-j|>ak\log k\}} q^{2|i-j|/k} \le \frac{2}{-\log q} k^{\log q^{2a} +1}
		\eea
		for all sufficiently large $k$.
	\end{proof}

	Now we prove Theorem \ref{main-theorem:bandableinv}.
	\begin{proof}[Proof of Theorem \ref{main-theorem:bandableinv}]
		Let $B_k(\Sigma_0)^{-1} = (w_{ij}^*) $. We have
		\bea
		\sum_{i\in I_j} |w_{ij}| &\le& \sum_{i\in I_j} |w_{ij}^*| + \sum_{i\in I_j} |w_{ij}-w_{ij}^*|  \\
		&\le& \sum_{i\in I_j} |w_{ij}^*|  + ||\Sigma_0^{-1} - B_k(\Sigma_0)^{-1} ||_1,
		\eea
		where $I_j = \{ i \in \{1,2,\ldots,p\} : |i-j|>ak\log k\}$. By Lemma \ref{main-lemma:l1bandinv}, there exists some positive constant $C_1$ depending only on $M$, $M_0$ and $M_1$ such that
		\bean
		||\Sigma_0^{-1} - B_k(\Sigma_0)^{-1} ||_1 \le C_1 k^{-\alpha},\label{formula:bandeableinv1}
		\eean
		for all sufficiently large $k$.
		
		For the upper bound of $\sum_{i\in I_j} |w_{ij}^*| $, Lemma \ref{lemma:bandinv} gives that, for all sufficiently large $k$ with $p > k\vee (ak\log k)$,
		\bean
		\max_j\sum_i \{ |w_{ij}^*|: |i-j|>ak\log k\} \le -\frac{2C_2}{\log q} k^{2a\log q +1},\label{formula:bandeableinv2}
		\eean
		where $q=(\kappa^{1/2}-1)/(\kappa^{1/2}+1)$, $C_2= \{||\Sigma^{-1}||\vee (1+\kappa^{1/2})^2\}/(2||\Sigma||)$ and $\kappa$ is the spectral condition number of $\Sigma$.
		By collecting the inequalities \eqref{formula:bandeableinv1} and \eqref{formula:bandeableinv2}, we complete the proof.
		 
	\end{proof}
	
	\section{Proof of Theorem 2}

	In this section, we prove Theorem \ref{main-theorem:taperupper} which gives the convergence rate of the tapering estimator. First, we present Lemma \ref{lemma:pdadjust} necessary for the proof of Theorem \ref{main-theorem:taperupper}.

	\begin{lemma}\label{lemma:pdadjust}
		Let $\Sigma$ and $\Sigma_0$ be $p\times p$-symmetric matrices and 
		$\Sigma^{(\epsilon_n)} = \Sigma + [\{\epsilon_n - \lambda_{\min}(\Sigma)\}\vee 0]I_p $. For all positive integer $r$ and all positive real number $\epsilon_n>0$,
		\bea
		||\Sigma^{(\epsilon_n)}-\Sigma_0||^r \le 2^{2r-1} ||\Sigma-\Sigma_0||^r + 4^{r-1} |\epsilon_n|^r.
		\eea
	\end{lemma}
	\begin{proof}
		By Lemma 4.14 in \cite{lee2020post}, we have 
		\bea
		\lambda_{\min}(\Sigma) &=& \lambda_{\min}(\Sigma-\Sigma_0+\Sigma_0)\\
		&\ge& \lambda_{\min}(\Sigma_0) - ||\Sigma-\Sigma_0||.
		\eea
		Using this inequality, we get
		\bean
		||\Sigma-\Sigma_0||&\ge& ||\Sigma-\Sigma_0||I( \lambda_{\min}(\Sigma)<0)\nonumber \\
		&\ge& \{||\Sigma-\Sigma_0|| -\lambda_{\min}(\Sigma_0)\} I( \lambda_{\min}(\Sigma)<0)\nonumber\\
		&\ge& -\lambda_{\min}(\Sigma)I( \lambda_{\min}(\Sigma)<0).\label{formula:lambdamin}
		\eean
		We also have 
		\bean
		\{\epsilon_n -\lambda_{\min}(\Sigma)\}\vee 0 &\le&
		\{\epsilon_n -\lambda_{\min}(\Sigma)I(\lambda_{\min}(\Sigma)<0)\}\vee 0 \nonumber\\
		&=&\epsilon_n -\lambda_{\min}(\Sigma)I(\lambda_{\min}(\Sigma)<0).\label{formula:lambdamin2}
		\eean
		Thus, we have
		\bea
		||\Sigma^{(\epsilon_n)}-\Sigma_0||^r &=& ||\Sigma-\Sigma_0 +[\{\epsilon_n - \lambda_{\min}(\Sigma)\}\vee 0]I_p    ||^r  \\
		&\le &
		2^{r-1}||\Sigma-\Sigma_0||^r + 
		2^{r-1}|\epsilon_n - \lambda_{\min}(\Sigma) I(\lambda_{\min}(\Sigma)<0)|^r\\
		&\le &2^{r-1}||\Sigma-\Sigma_0||^r + 4^{r-1}|\epsilon_n|^r +   4^{r-1}|-\lambda_{\min}(\Sigma) I(\lambda_{\min}(\Sigma)<0)|^r \\
		&\le& 2^{r-1}||\Sigma-\Sigma_0||^r + 4^{r-1}|\epsilon_n|^r + 4^{r-1} ||\Sigma-\Sigma_0||^r\\
		&\le& 2^{2r-1} ||\Sigma-\Sigma_0||^r + 4^{r-1} |\epsilon_n|^r,
		\eea
		where the first and third inequalities are satisfied inequalities \eqref{formula:lambdamin2} and \eqref{formula:lambdamin}, respectively.	 
	\end{proof}

	\begin{proof}[Proof of Theorem \ref{main-theorem:taperupper}]

		
		We have
		\bean
		&&E(|| \Sigma_{0,YX} \Sigma_{0,XX}^{-1} -  \psi\{T_k^{(\epsilon_n)}(S_n)\} ||^2) \nonumber\\ 
		&\le& 2E(||T_k^{(\epsilon_n)}(S_n)_{YX}-\Sigma_{0,YX}||^2 || T_k^{(\epsilon_n)}(S_n)_{XX}^{-1}||^2) \nonumber\\
		&&+ 2||\Sigma_{0,YX}||^2E(||T_k^{(\epsilon_n)}(S_n)_{XX}^{-1}-\Sigma_{0,XX}^{-1}||^2 ) \nonumber \\
		&\le& \frac{2^3}{M_1^2}E(||T_k^{(\epsilon_n)}(S_n)_{YX}-\Sigma_{0,YX}||^2
		I[\lambda_{\min}\{T_k(S_n)_{XX}\} >M_1/2]) \label{formula:taperupper1}\\
		&&+\frac{2}{\epsilon_n^2}E(||T_k^{(\epsilon_n)}(S_n)_{YX}-\Sigma_{0,YX}||^2 
		I[\lambda_{\min}\{T_k(S_n)_{XX}\} \le M_1/2])\label{formula:taperupper2}\\
		&&+ 2M_0^2E(||T_k^{(\epsilon_n)}(S_n)_{XX}^{-1}-\Sigma_{0,XX}^{-1}||^2 ),\label{formula:taperupper3}
		\eean
		where the last inequality holds since $\lambda_{\min}\{T_k^{(\epsilon_n)}(S_n)_{XX}\}\ge \lambda_{\min}\{T_k^{(\epsilon_n)}(S_n)\}\ge\epsilon_n$.
		We have that there exists some positive constant $C_1$ depending only on $M$, $M_0$, $M_1$ and $\alpha$ such that
		\bea
		\eqref{formula:taperupper1}
		&\le& \frac{2^3}{M_1^2}E(||T_k^{(\epsilon_n)}(S_n)-\Sigma_{0}||^2) \\
		&\le& \frac{2^3}{M_1^2}\{2^3 E(||T_k(S_n)-\Sigma_{0}||^2) + 4 \epsilon_n^2 \}\\
		&\le&  C_1 \Big\{  \frac{k + \log p}{n} + k^{-2\alpha} + \epsilon_n^2\Big\},
		\eea
		for all sufficiently large $n$. The second and last inequalities hold by Lemma \ref{lemma:pdadjust} and Theorem 2 of \cite{cai2010optimal}, respectively.
		By Lemmas \ref{main-lemma:evminfreq} and \ref{main-lemma:upper1}, there exist some positive constants $C_2$ and $\lambda_1$ depending only on $M_0$ and $M_1$ such that
		\bea
		\eqref{formula:taperupper2}
		&\le& 
		\frac{2}{\epsilon_n^2}E(||T_k^{(\epsilon_n)}(S_n)-\Sigma_{0}||^4)^{1/2} P (\lambda_{\min}(T_k(S_n)_{XX}) \le M_1/2)^{1/2}\\
		&\le& \frac{C_2 }{\epsilon_n^2}   p^{1/2}5^{k/2}\exp(-\lambda_1 n),
		\eea
		for all sufficiently large $n$.
		Finally, by Lemma \ref{main-lemma:upperblockinv}, we get the upper bound of \eqref{formula:taperupper3} as 
		\bea
		\eqref{formula:taperupper3} \le 
		C_3 \Big(  \frac{k + \log p}{n} + k^{-2\alpha} + \epsilon_n^2 + \frac{p^{1/2}5^{k/2} \exp(-\lambda_2 n)}{\epsilon_n^2}\Big),
		\eea
		for some positive constants $C_3$ and $\lambda_2$ depending only on $M$, $M_0$, $M_1$ and $\alpha$. Combining the upper bounds of \eqref{formula:taperupper1}, \eqref{formula:taperupper2} and \eqref{formula:taperupper3}, we complete the proof.	 
		
	\end{proof}

	\section{Proof of Lemma 3}
	In this section, we prove the Lemma \ref{main-lemma:trueerror}. For this proof, we present Lemma \ref{lemma:bandblockinv}.
	
	\begin{lemma}\label{lemma:bandblockinv}
		Let $q$ and $k$ be positive integers with $k<q$, and $A\in\calC_q$ be a $k$-band matrix. For positive real numbers $a$ and $b$ with $q >\lfloor ak\log k\rfloor +\lfloor bk\log k\rfloor $, we express $A$ and $A^{-1}$ as
		\bea
		A=
		\begin{bmatrix}
			A_{11} & A_{12}\\
			A_{21} & A_{22}
		\end{bmatrix},
		A^{-1} =\Omega = 
		\begin{bmatrix}
			\Omega_{11} & \Omega_{12}\\
			\Omega_{21} & \Omega_{22}
		\end{bmatrix} ,
		\eea
		where $A_{22},\Omega_{22}\in \bbR^{(\lfloor ak\log k\rfloor +\lfloor bk\log k\rfloor  )\times (\lfloor ak\log k\rfloor  + \lfloor bk\log k\rfloor )}$.
		There exist some positive constants $\lambda$ and $C$ depending only on $||A||$ and $||A^{-1}||$ such that 
		\bea
		||M_{\lfloor ak\log k\rfloor +1}^{(\lfloor bk\log k \rfloor )}(\Omega_{22}) -M_{\lfloor ak\log k\rfloor +1}^{(\lfloor bk\log k\rfloor )} (A_{22}^{-1})  ||\le C k^{-a\lambda+1},
		\eea
		for all sufficiently large $k$ and $q$ with $\lfloor ak\log k \rfloor \ge k$ and $q >\lfloor ak\log k\rfloor +\lfloor bk\log k\rfloor $.
		
	\end{lemma}
	\begin{proof}
		Let
		$\{B_{\lfloor ak\log k\rfloor }(\Omega)\}_{12}= \{B_{\lfloor ak\log k\rfloor }(\Omega)\}_{1:q^*,(q^*+1):q},$
		where $q^* = q- \lfloor ak\log k\rfloor - \lfloor bk\log k\rfloor $ and $B_{\lfloor ak\log k\rfloor }$ is $\lfloor ak\log k\rfloor $-banding operator defined in \cite{bickel2008regularized}.
		First, we show the equality
		\bean
		M_{\lfloor ak\log k\rfloor +1}^{(\lfloor bk\log k\rfloor )}(A_{22}^{-1}[I-A_{21} \{B_{\lfloor ak\log k\rfloor }(\Omega)\}_{12}]) = M_{\lfloor ak\log k\rfloor +1}^{(\lfloor bk\log k\rfloor )}(A_{22}^{-1}). \label{formula:block2} 
		\eean
		Since $\lfloor ak\log k\rfloor  \ge k$ and $A$ is a $k$-band matrix, $A_{21}$ is expressed as 
		\[
		A_{21}=\left[
		\begin{matrix} O_{\{\lfloor ak\log k\rfloor \times (q^*-\lfloor ak\log k\rfloor )\}} &A^*\\  O_{\{\lfloor bk\log k\rfloor \times (q^*-\lfloor ak\log k\rfloor )\}} & O_{(\lfloor bk\log k\rfloor \times \lfloor ak\log k\rfloor )}  \end{matrix}
		\right],
		\]
		for some $\lfloor ak\log k\rfloor \times \lfloor ak\log k\rfloor $-matrix $A^*$.  
		Likewise, $\{B_{\lfloor ak\log k\rfloor }(\Omega)\}_{12}$ is expressed as 
		\[
		\{B_{\lfloor ak\log k\rfloor }(\Omega)\}_{12}=\left[
		\begin{matrix} O_{\{(q^*-\lfloor ak\log k\rfloor )\times \lfloor ak\log k\rfloor  \}}  & O_{\{(q^*-\lfloor ak\log k\rfloor )\times \lfloor bk\log k\rfloor \}}  \\
			B^*_{(\lfloor ak\log k\rfloor  \times \lfloor ak\log k\rfloor )} & O_{(\lfloor ak\log k\rfloor \times \lfloor bk\log k\rfloor )}  \end{matrix}
		\right],
		\]
		for some $\lfloor ak\log k\rfloor \times \lfloor ak\log k\rfloor $-matrix $B^*$.
		Thus, 
		\bea
		A_{21} \{B_{\lfloor ak\log k\rfloor }(\Omega)\}_{12} = \left[
		\begin{matrix} A^* B^*  & O_{(\lfloor ak\log k\rfloor  \times \lfloor bk\log k\rfloor  )}  \\
			O_{(\lfloor bk\log k\rfloor  \times \lfloor ak\log k\rfloor )} & O_{(\lfloor bk\log k\rfloor \times \lfloor bk\log k\rfloor )}  \end{matrix}
		\right], 
		\eea
		and 
		$M_{\lfloor ak\log k\rfloor +1}^{(\lfloor bk\log k\rfloor )}(A_{22}^{-1}[A_{21} \{B_{\lfloor ak\log k\rfloor }(\Omega)\}_{12}])$ is the zero matrix, which gives the equality \eqref{formula:block2}.
		
		Next, we show the upper bound of 
		\bean
		|| M_{\lfloor ak\log k\rfloor +1}^{(\lfloor bk\log k\rfloor )}(A_{22}^{-1}[I-A_{21} \{B_{\lfloor ak\log k\rfloor }(\Omega)\}_{12}]) - M_{\lfloor ak\log k\rfloor +1}^{(\lfloor bk\log k\rfloor )}(\Omega_{22})
		|| \label{formula:block3} .
		\eean
		By the\ block matrix inversion formula, $\Omega_{22} = A_{22}^{-1} (I- A_{21}\Omega_{12}) $, we have
		\bea
		\eqref{formula:block3}&=& ||M_{\lfloor ak\log k\rfloor +1}^{(\lfloor bk\log k\rfloor )}(A_{22}^{-1}A_{21}[ \{B_{\lfloor ak\log k\rfloor }(\Omega)\}_{12} - \Omega_{12}])||\\   
		&\le& 
		||
		A_{22}^{-1}A_{21}||~ || \{B_{\lfloor ak\log k\rfloor }(\Omega)\}_{12} - \Omega_{12}
		||\nonumber\\
		&\le& 
		\lambda_{\min}(A_{22})^{-1}||A_{21}||~|| B_{\lfloor ak\log k\rfloor }(\Omega) - \Omega ||_1\nonumber\\
		&\le&||A|| ~ ||A^{-1}||~|| B_{\lfloor ak\log k\rfloor }(\Omega) - \Omega ||_1.
		\eea
		By Lemma \ref{lemma:bandinv}, there exist some positive constants $\lambda_1$ and $C_1$ depending only on $||A||$ and $||A^{-1}||$ such that
		$$|| B_{\lfloor ak\log k\rfloor }(\Omega) - \Omega ||_1 \le C_1 k^{-a\lambda_1+1},$$  
		for all sufficiently large $k$ with $p > k\vee (ak\log k)$.
		By combining this inequality with the equality \eqref{formula:block2}, 
		we get
		\bea
		||M_{\lfloor ak\log k\rfloor +1}^{(\lfloor bk\log k\rfloor )}(A_{22}^{-1}) -  M_{\lfloor ak\log k\rfloor +1}^{(\lfloor bk\log k\rfloor )}(\Omega_{22}) || \le C_1 ||A|| ~||A^{-1}|| k^{-a\lambda_1+1},
		\eea
		for all sufficiently large $k$. 	 
	\end{proof}

	Next, we prove Lemma \ref{main-lemma:trueerror}.
	\begin{proof}[Proof of Lemma \ref{main-lemma:trueerror}]
		There exist some positive constants $C_1$ and $C_2$ depending only on $M$, $M_0$ and $M_1$ such that
		\bea
		&&|| \psi(\Sigma_0)  - T_k(\Sigma_0)_{YX} T_k(\Sigma_{0,XX})^{-1}|| \\ &\le& 
		|| \Sigma_{0,XX}^{-1}|| ~|| T_k(\Sigma_0)_{YX} - \Sigma_{0,YX} || \\
		&&+ ||T_k(\Sigma_0)_{YX} ||~ || \Sigma_{0,XX}^{-1}-T_k(\Sigma_{0,XX})^{-1}||\\
		&\le& M_1^{-1} M(\lfloor k/2\rfloor)^{-\alpha} + 
		(||\Sigma_0|| + ||T_k(\Sigma_0) -\Sigma_0|| )
		|| \Sigma_{0,XX}^{-1}-T_k(\Sigma_{0,XX})^{-1}||\\
		&\le&M_1^{-1} M(\lfloor k/2\rfloor)^{-\alpha} + \{M_0 + M(\lfloor k/2\rfloor)^{-\alpha}\}C_1 (\lfloor k/2\rfloor)^{-\alpha}\\
		&\le& C_2 (\lfloor k/2\rfloor)^{-\alpha},
		\eea
		for all sufficiently large $k$. The third inequality holds by Lemma \ref{main-lemma:l1bandinv} since $\Sigma_{0,XX}\in\calF_{p_0,\alpha}(M,M_0,M_1)$.
		Then, we have 
		\bea
		&&   || \psi(\Sigma_0)  -  T_k(\Sigma_0)_{YX} \Lambda^{(0)}\{T_k(\Sigma_{0,XX});  2\lfloor ak\log k \rfloor\}|| \nonumber\\
		&\le& || \psi(\Sigma_0)  - T_k(\Sigma_0)_{YX} T_k(\Sigma_{0,XX})^{-1}|| \\
		&&+ || T_k(\Sigma_0)_{YX} [\Lambda^{(0)}\{T_k(\Sigma_{0,XX});  2\lfloor ak\log k \rfloor\}- T_k(\Sigma_{0,XX})^{-1}]|| \\
		&\le&   C_2 (\lfloor k/2\rfloor )^{-\alpha } + || T_k(\Sigma_0)_{YX} [\Lambda^{(0)}\{T_k(\Sigma_{0,XX});  2\lfloor ak\log k \rfloor\}- T_k(\Sigma_{0,XX})^{-1}]|| .
		\eea
		For this upper bound, we have
		\bean
		&& || T_k(\Sigma_0)_{YX} [\Lambda^{(0)}\{T_k(\Sigma_{0,XX});  2\lfloor ak\log k \rfloor\}- T_k(\Sigma_{0,XX})^{-1}]|| \nonumber\\
		&\le& || T_k(\Sigma_{0})_{YX} [  T_k(\Sigma_{0,XX})^{-1}  -M_{p_0-\lfloor ak\log k\rfloor+1}^{*(\lfloor ak\log k \rfloor)}\{T_k(\Sigma_{0,XX})^{-1}\}]||   \label{formula:trueerror3} \\
		&&+||T_k(\Sigma_{0})_{YX}[M_{p_0-\lfloor ak\log k\rfloor+1}^{*(\lfloor ak\log k\rfloor )}\{T_k(\Sigma_{0,XX})^{-1}\} -\Lambda^{(0)}\{T_k(\Sigma_{0,XX}); 2\lfloor ak\log k\rfloor \}]||,\label{formula:trueerror4}
		\eean
		
		First, we show the upper bound of \eqref{formula:trueerror3}.
		Since $T_k(\Sigma_0)$ is a $k$-band matrix, $T_k(\Sigma_0)_{YX} $ is expressed as 
		\bean
		T_k(\Sigma_{0})_{YX}=\left[
		\begin{matrix} O_{k\times (p_0-k)} &\{T_k(\Sigma_0)_{YX}\}_{1:k,(p_0-k+1):p_0}\\
			O_{(p-p_0-k)\times (p_0-k)} & O_{(p-p_0-k)\times k} \end{matrix}
		\right]. \label{formula:Tk21}
		\eean
		Since $\lfloor ak\log k\rfloor  \ge k$, the definition of $M_{p_0-\lfloor ak\log k\rfloor +1}^{*(\lfloor ak\log k\rfloor )}$ gives
		\bea
		&&T_k(\Sigma_{0,XX})^{-1}  -M_{p_0-\lfloor ak\log k\rfloor +1}^{*(\lfloor ak\log k\rfloor )}\{T_k(\Sigma_{0,XX})^{-1}\}\\
		&=&\left[ 
		\begin{matrix} B^* & C^*\\
			\{T_k(\Sigma_{0,XX})^{-1}\}_{ (p_0-k+1) : p_0, 1:(p_0-\lfloor ak\log k\rfloor ) } & O_{k\times \lfloor ak\log k\rfloor } \end{matrix}
		\right],
		\eea
		for some $(p_0-k)\times (p_0-\lfloor ak\log k\rfloor )$-matrix $B^*$ and some $(p_0-k)\times \lfloor ak\log k\rfloor $-matrix $C^*$.
		Thus, 
		\bea
		&&T_k(\Sigma_{0})_{YX}[T_k(\Sigma_{0,XX})^{-1}  -M_{p_0-\lfloor ak\log k\rfloor +1}^{*(\lfloor ak\log k\rfloor )}\{T_k(\Sigma_{0,XX})^{-1}\}]\\
		&=&\left[ 
		\begin{matrix} \{T_k(\Sigma_0)_{YX}\}_{1:k,(p_0-k+1):p_0} \{T_k(\Sigma_{0,XX})^{-1}\}_{ (p_0-k+1) : p_0, 1:(p_0-\lfloor ak\log k\rfloor ) } & O_{k\times \lfloor ak\log k\rfloor}\\
			O_{(p-p_0-k)\times (p_0-\lfloor ak\log k\rfloor )} & O_{(p-p_0-k)\times \lfloor ak\log k\rfloor } \end{matrix}
		\right],
		\eea
		and
		\bea
		\eqref{formula:trueerror3}
		&= & ||\{T_k(\Sigma_0)_{YX}\}_{1:k,(p_0-k+1):p_0}  \{T_k(\Sigma_{0,XX})^{-1}\}_{ (p_0-k+1) : p_0, 1:(p_0-\lfloor ak\log k\rfloor ) }||\\
		&\le&||T_k(\Sigma_0)_{YX} ||~ ||\{T_k(\Sigma_{0,XX})^{-1}\}_{ (p_0-k+1) : p_0, 1:(p_0-\lfloor ak\log k\rfloor ) }||.
		\eea
		Let $\{T_k(\Sigma_{0,XX})\}^{-1}=(b_{ij})_{1 \le i,j \le p_0}$. Every element $b_{ij}$ in $\{T_k(\Sigma_{0,11})^{-1}\}_{ (p_0-k+1) : p_0, 1:(p_0-\lfloor ak\log k\rfloor ) }$ satisfies 
		\bea
		|i-j|&\ge& \lfloor  ak\log k\rfloor -k+1 \\
		&>& \lfloor  ak\log k\rfloor/2,
		\eea 
		since $\lfloor  ak\log k\rfloor/2\ge k$.
		Thus, by Lemma \ref{lemma:bandinv}, there exist some positive constants $C_3$ and $\lambda_1$ depending only on $M_0$ and $M_1$ such that
		\bea
		&&||\{T_k(\Sigma_{0,XX})^{-1}\}_{ (p_0-k+1) : p_0, 1:(p_0-\lfloor ak\log k\rfloor ) }||_2\\
		&\le&
		||\{T_k(\Sigma_{0,XX})^{-1}\}_{ (p_0-k+1) : p_0, 1:(p_0-\lfloor ak\log k\rfloor ) } ||_1 \vee ||\{T_k(\Sigma_{0,XX})^{-1}\}_{ (p_0-k+1) : p_0, 1:(p_0-\lfloor ak\log k\rfloor ) } ||_\infty \\
		&\le& C_3 k^{-a\lambda_1 +1},
		\eea
		for all sufficiently large $k$ with $p_0>k \vee \lfloor ak\log k\rfloor /2$. Since 
		$$||T_k(\Sigma_0)_{YX} ||\le ||\Sigma_0|| + ||\Sigma_0-T_k(\Sigma_0)||\le    M_0 + M\lfloor k/2\rfloor,$$ we get
		\bean
		\eqref{formula:trueerror3} \le (M_0 + M\lfloor k/2\rfloor)C_2 k^{-a\lambda_1 +1},\label{formula:trueerror5}
		\eean
		for all sufficiently large $k$ with $p_0>k \vee \lfloor ak\log k\rfloor /2$.
		
		Next, we show the upper bound of \eqref{formula:trueerror4}.
		Since $T_k(\Sigma_0)_{YX}$ is expressed as \eqref{formula:Tk21}, we have
		\bean
		\eqref{formula:trueerror4}
		&\le& ||T_k(\Sigma_0)|| ~||[M_{p_0-\lfloor ak\log k\rfloor+1}^{*(\lfloor ak\log k\rfloor )}\{T_k(\Sigma_{0,XX})^{-1}\} -\Lambda^{(0)}\{T_k(\Sigma_{0,XX}); 2\lfloor ak\log k\rfloor \}]_{(p_0-k+1):p_0,1:p_0}||.\label{formula:trueerror7}
		\eean
		Note
		\bea
		&&[M_{p_0-\lfloor ak\log k\rfloor+1}^{*(\lfloor ak\log k\rfloor )}\{T_k(\Sigma_{0,XX})^{-1}\} -\Lambda^{(0)}\{T_k(\Sigma_{0,XX}); 2\lfloor ak\log k\rfloor \}]_{(p_0-k+1):p_0,1:p_0}\\
		&=&\left[ 
		\begin{matrix} 
			O_{ k\times (p_0-2 \lfloor ak\log k \rfloor)} & -D_1^* & D_3^*-D_2^*
		\end{matrix}
		\right],
		\eea
		where $D_1^*$, $D_2^*$ and $D_3^*$ are $k\times \lfloor ak\log k \rfloor$-matrices with 
		\bea
		D_1^* &=& [M_{p_0- 2\lfloor ak\log k \rfloor+1}^{( 2\lfloor ak\log k \rfloor)}\{T_k(\Sigma_{0,XX} )\}^{-1}]_{( 2\lfloor ak\log k \rfloor-k+1): 2\lfloor ak\log k \rfloor,1:\lfloor ak\log k \rfloor}\\
		D_2^* &=& [M_{p_0- 2\lfloor ak\log k \rfloor+1}^{( 2\lfloor ak\log k \rfloor)}\{T_k(\Sigma_{0,XX} )\}^{-1}]_{( 2\lfloor ak\log k \rfloor-k+1): 2\lfloor ak\log k \rfloor,(\lfloor ak\log k \rfloor+1): 2\lfloor ak\log k \rfloor}\\
		D_3^* &=& \{T_k(\Sigma_{0,XX})^{-1}\}_{(p_0-k+1) : p_0,(p_0-\lfloor ak\log k \rfloor+1):p_0 }.
		\eea
		We have
		\bea
		\eqref{formula:trueerror4}
		&\le& ||T_k(\Sigma_0)|| ( ||D_1^*|| + ||D_2^* - D_3^*||  ).
		\eea
		First, we show the upper bound of $||D_2^* - D_3^*||$. 
		Let 
		\bea
		A_{22}&=& M_{p_0- 2\lfloor ak\log k \rfloor+1}^{( 2\lfloor ak\log k \rfloor)}\{T_k(\Sigma_{0,XX} )\}\\
		\Omega_{22} &=& M_{p_0- 2\lfloor ak\log k \rfloor+1}^{( 2\lfloor ak\log k \rfloor)}\{T_k(\Sigma_{0,XX} )^{-1}\}.
		\eea
		By Lemma \ref{lemma:bandblockinv}, there exist some positive constants $C_4$ and $\lambda_2$ depending only on $M_0$ and $M_1$ such that
		\bea
		||D_2^* - D_3^*|| &\le&  ||  M_{\lfloor ak\log k \rfloor+1}^{(\lfloor ak\log k\rfloor )}(\Omega_{22}) -M_{\lfloor ak\log k \rfloor+1}^{(\lfloor ak\log k \rfloor)} (A_{22}^{-1})   ||   \\
		&\le& C_3 k^{-a\lambda_2 +1},
		\eea
		for all sufficiently large $k$ with $\lfloor ak\log k\rfloor \ge k$ and $p_0 >2 \lfloor ak\log k\rfloor$.
		
		Next, we show the upper bound of $||D_1^*||$.
		Let $M_{p_0- 2\lfloor ak\log k \rfloor+1}^{( 2\lfloor ak\log k \rfloor)}\{T_k(\Sigma_{0,XX} )\}^{-1}=(c_{ij})_{1\le i,j\le 2\lfloor ak\log k \rfloor}$. Every element $c_{ij}$ in $D_1^*$ satisfies 
		\bea
		|i-j|&\ge& \lfloor  ak\log k\rfloor -k+1 \\
		&\ge& \lfloor  ak\log k\rfloor/2 ,
		\eea 
		since $\lfloor  ak\log k\rfloor/2\ge k$.
		Thus, Lemma \ref{lemma:bandinv} gives that there exist some positive constants $C_5$ and $\lambda_3$ depending only on $M_0$ and $M_1$ such that
		\bea
		||D_1^* || \le C_5 k^{-a\lambda_3+1},
		\eea
		for all sufficiently large $k$ with $p_0 > k\vee (ak\log k)$. By combining this inequality with \eqref{formula:trueerror7}, we have     
		\bean
		\eqref{formula:trueerror4} \le (M_0 + M(\lfloor k/2 \rfloor)^{-\alpha}) (  C_2 k^{-a\lambda_2 +1}  + C_3 k^{-a\lambda_3+1}) \label{formula:trueerror6}.  
		\eean
		By collecting the inequalities \eqref{formula:trueerror5} and \eqref{formula:trueerror6}, 
		the proof is completed.	 
	\end{proof}

	\section{Proofs of Theorems 6 and 7}\label{sec:thm6}
	
	In this section, we prove Theorems \ref{main-theorem:TPPPupper} and \ref{main-theorem:PPPupper} which give the P-risk convergence rates of the tapering and blockwise tapering post-processed posteriors, respectively. 
	First, we present Lemmas \ref{lemma:upper2}-\ref{lemma:invfactor} necessary for the proofs of Theorems \ref{main-theorem:TPPPupper} and \ref{main-theorem:PPPupper}. 

	\begin{lemma}\label{lemma:upper2}
		Let $k$ and $p$ be positive integers with $k\le p$.
		Suppose $\Sigma_0\in\calF_{p,\alpha}(M,M_0,M_1)$, and let the prior $\pi^i$ of $\Sigma$ be $IW_p(A_n,\nu_n)$ for $A_n\in\calC_p$ and $\nu_n>2p$.
		If $k \vee ||A_n|| \vee (\nu_n-2p)\vee\log p=o(n) $ and $\epsilon_n=O(1)$, then for all sufficiently large $n$,
		$E_{\Sigma_{0}}\{E^{\pi^i}(||T_k^{(\epsilon_n)}(\Sigma)-\Sigma_0||^4\mid\bbZ_n)\}$ is bounded above by some positive constant depending only on $M$, $M_0$, $M_1$ and $\alpha$.
		
	\end{lemma}
	The proof of this lemma is given in Section \ref{sec:proofmis}.

	\begin{lemma}\label{lemma:evmin}
		Suppose the same setting of Lemma \ref{lemma:upper2}.
		If $c \le \lambda_{\min}(\Sigma_0)/2$, $\lfloor k/2\rfloor> \{4M/\lambda_{\min}(\Sigma_0)\}^{1/\alpha}$ and $k \vee ||A_n|| \vee (\nu_n-2p)\vee\log p=o(n) $, then there exist some positive constants $C$ and $\lambda$ depending only on $M_0$ and $M_1$ such that
		\bea
		E_{\Sigma_0}(P^{\pi^i}[\lambda_{\min}\{T_k(\Sigma)\} \le c  \mid \bbZ_n  ])\le 
		C p 5^k \exp(-\lambda n) ,
		\eea
		for all sufficiently large $n$.
	\end{lemma}
	The proof of this lemma is given in Section \ref{sec:proofmis}.
	
	\begin{lemma}\label{lemma:TPPPupper}
		Suppose the same setting of Lemma \ref{lemma:upper2}.
		If $||A_n||\vee(\nu_n-2p)\vee k\vee\log p=o(n)$ and $\epsilon_n=O(1)$, then there exist some positive constant $C$ depending only on $M$, $M_0$, $M_1$ and $\alpha$ such that  
		\bea
		E_{\Sigma_0}\{ E^{\pi^i} ( ||\Sigma_0 -  T_k^{(\epsilon_n)}(\Sigma)||^2\mid\bbZ_n)\} 
		&\le& C\Big( k^{-2\alpha } + \frac{k+\log p }{n} +\epsilon_n^2  \Big),
		\eea
		for sufficiently large $n$. 
	\end{lemma}
	The proof of this lemma is given in Section \ref{sec:proofmis}.

	\begin{lemma}\label{lemma:invfactor}
		Suppose the same setting of Lemma \ref{lemma:upper2}. Let $q$ and $k$ be positive constants with $k<q\le p$, and let $\{T_k^{(\epsilon_n)}(\Sigma)\}_{1:q,1:q}$ and $(\Sigma_0)_{1:q,1:q}$ be denoted by $T_k^{(\epsilon_n)}(\Sigma)_{11}$ and $\Sigma_{0,11}$, respectively.
		If $\lfloor k/2 \rfloor > \{4M/\lambda_{\min}(\Sigma_0)\}^{1/\alpha}$ and $k \vee ||A_n|| \vee (\nu_n-2p)\vee\log p=o(n) $, then there exist positive constants $C$ and $\lambda$ depending only on $M$, $M_0$, $M_1$ and $\alpha$ such that 
		\bea
		E_{\Sigma_0}\{ E^{\pi^i}(||\Sigma_{0,11}^{-1} - T_k^{(\epsilon_n)}(\Sigma)_{11}^{-1}||^2\mid \bbZ_n)\}\le
		C\Big\{  \frac{k+\log p}{n} + k^{-2\alpha} + \epsilon_n^2 +\frac{p^{1/2} 5^{k/2} \exp(-\lambda n)}{\epsilon_n^2}  \Big\},
		\eea
		for all sufficiently large $n$.
	\end{lemma}
	The proof of this lemma is given in Section \ref{sec:proofmis}.

	Now we prove Theorems \ref{main-theorem:TPPPupper} and \ref{main-theorem:PPPupper}.
	
	\begin{proof}[Proof of Theorem \ref{main-theorem:TPPPupper}]
		
		Since $\lambda_{\min}\{T_k^{(\epsilon_n)}(\Sigma)_{XX}^{-1}\}\ge\lambda_{\min}\{T_k^{(\epsilon_n)}(\Sigma)^{-1}\}\ge \epsilon_n$, we have 
		\bean
		&&E_{\Sigma_0}\{ E^{\pi^i}(|| \psi(\Sigma_0)  -  \psi\{T_k^{(\epsilon_n)}(\Sigma)\} ||^2\mid\bbZ_n)\}\nonumber\\
		&=& E_{\Sigma_0}\{ E^{\pi^i}(||\Sigma_{0,YX} \Sigma_{0,XX}^{-1} - T_k^{(\epsilon_n)}(\Sigma)_{YX} T_k^{(\epsilon_n)}(\Sigma)_{XX}^{-1}||^2\mid\bbZ_n)\}\nonumber\\
		&\le& 2 E_{\Sigma_0}\{ E^{\pi^i}( ||T_k^{(\epsilon_n)}(\Sigma)_{YX}-\Sigma_{0,YX}||^2||T_k^{(\epsilon_n)}(\Sigma)_{XX}^{-1}||^2\mid\bbZ_n)\}\nonumber\\
		&&+2||\Sigma_{0,YX}||^2 E_{\Sigma_0}\{ E^{\pi^i}(  ||   T_k^{(\epsilon_n)}(\Sigma)_{XX}^{-1}- \Sigma_{0,XX}^{-1}||^2\mid\bbZ_n)\} \nonumber\\
		&\le& \frac{2^3}{M_1^2} E_{\Sigma_0}\{ E^{\pi^i}( ||T_k^{(\epsilon_n)}(\Sigma)_{YX}-\Sigma_{0,YX}||^2\mid\bbZ_n)\}\label{formula:thmTPPPupper1}\\
		&&+ \frac{2}{\epsilon_n^2}E_{\Sigma_0}\{ E^{\pi^i}(||T_k^{(\epsilon_n)}(\Sigma)_{YX}-\Sigma_{0,YX}||^2 I[\lambda_{\min}\{T_k(S_n)\}\le M_1/2]\mid\bbZ_n)\}\label{formula:thmTPPPupper2}\\
		&&+2||\Sigma_{0,YX}||^2 E_{\Sigma_0}\{ E^{\pi^i}(  ||   T_k^{(\epsilon_n)}(\Sigma)_{XX}^{-1}- \Sigma_{0,XX}^{-1}||^2\mid\bbZ_n)\}. \label{formula:thmTPPPupper3}
		\eean
		Lemmas \ref{lemma:TPPPupper} and \ref{lemma:invfactor} give that there exist some positive constants $C_1$ and $\lambda_1$ depending only on $M$, $M_0$, $M_1$ and $\alpha$ such that 
		\bea
		\eqref{formula:thmTPPPupper1} + \eqref{formula:thmTPPPupper3} \le C_1 \Big\{  \frac{k+\log p}{n} + k^{-2\alpha} + \epsilon_n^2 +\frac{p^{1/2} 5^{k/2} \exp(-\lambda_1 n)}{\epsilon_n^2}  \Big\},    
		\eea
		for all sufficiently large $n$. 
		Next we show the upper bound of \eqref{formula:thmTPPPupper2}. By Lemmas \ref{lemma:upper2} and \ref{lemma:evmin}, there exist positive constants $C_2$ and $\lambda_2$ depending only on $M$, $M_0$, $M_1$ and $\alpha$ such that 
		\bea
		\eqref{formula:thmTPPPupper2}
		&\le& \frac{2}{\epsilon_n^2}E_{\Sigma_0}\{ E^{\pi^i}(||T_k^{(\epsilon_n)}(\Sigma)_{YX}-\Sigma_{0,YX}||^4\mid\bbZ_n)\}^{1/2}\\
		&&\times  E_{\Sigma_0}\{ P^{\pi^i}(I[\lambda_{\min}\{T_k(S_n)\}\le M_1/2]\mid\bbZ_n)\}^{1/2}\\
		&\le& C_2  \frac{p^{1/2} 5^{k/2}}{\epsilon_n^2} \exp(-\lambda_2 n),
		\eea
		for all sufficiently large $n$.
		Collecting the upper bounds of \eqref{formula:thmTPPPupper1}, \eqref{formula:thmTPPPupper2} and \eqref{formula:thmTPPPupper3}, we complete the proof.	 
		
	\end{proof}

	\begin{proof}[Proof of Theorem~\ref{main-theorem:PPPupper}]
		
		We have
		\bean
		&&E_{\Sigma_0}\{E^{\pi^i}(|| \psi(\Sigma_0)  - T_k(\Sigma)_{YX} \Lambda^{(\epsilon_n)}\{ T_k(\Sigma)_{XX}; 2\lfloor ak\log k \rfloor\}||^2\mid\bbZ_n)\} \nonumber\\
		&\le& 2||\psi(\Sigma_0) - T_k(\Sigma_0)_{YX}\Lambda^{(0)}\{T_k(\Sigma_{0,XX});2\lfloor ak\log k \rfloor\}||^2 \label{formula:PPPupper1}\\
		&&+ 4E_{\Sigma_0}\{E^{\pi^i}(||\{T_k(\Sigma_{0})_{YX}- T_k(\Sigma)_{YX}\} \Lambda^{(\epsilon_n)}\{ T_k(\Sigma)_{XX};2\lfloor ak\log k \rfloor\} ||^2\mid\bbZ_n)\}\label{formula:PPPupper2}\\
		&&+ 4|| T_k(\Sigma_0)_{YX}||^2E_{\Sigma_0}\{E^{\pi^i}( \Lambda^{(0)}\{T_k(\Sigma_{0,XX});2\lfloor ak\log k \rfloor\}\nonumber \\
		&&-\Lambda^{(\epsilon_n)}\{ T_k(\Sigma)_{XX};2\lfloor ak\log k \rfloor\}||^2\mid\bbZ_n)\}.\label{formula:PPPupper3}
		\eean
		Lemma \ref{main-lemma:trueerror} gives that there exist positive constants $C_1$ and $\lambda_1$ depending only on $M$, $M_0$, $M_1$ and $\alpha$ such that
		\bea
		\eqref{formula:PPPupper1}  \le C_1 k^{-2\{\alpha\wedge (a\lambda_1-1)\}}
		\eea
		for all sufficiently large $k$ with $\lfloor ak\log k\rfloor /2 >k$ and $p_0> 2\lfloor ak\log k\rfloor $.
		Let $\tilde{M} = M_{p_0-2\lfloor ak\log k \rfloor+1}^{(2\lfloor ak\log k \rfloor)}$ in this proof. By the definition of $\Lambda^{(\epsilon_n)}$, we have  
		\bea
		|| \Lambda^{(\epsilon_n)}\{T_k(\Sigma_{XX});2\lfloor ak\log k \rfloor\}|| = || T_k^{(\epsilon_n)}\{\tilde{M} ( \Sigma_{XX})\}^{-1}||.
		\eea
		Thus, there exist  positive constants $C_2$ and $\lambda_2$ depending only on $M$, $M_0$, $M_1$ and $\alpha$ such that
		\bea
		\eqref{formula:PPPupper2}
		&\le& 4E_{\Sigma_0}\{E^{\pi^i}(||T_k(\Sigma_{0})_{YX}- T_k(\Sigma)_{YX}||^2 || T_k^{(\epsilon_n)}\{\tilde{M} ( \Sigma_{XX})\}^{-1}||^2\mid\bbZ_n)\}   \\
		&\le& \frac{4}{\epsilon_n^2}E_{\Sigma_0}\{E^{\pi^i}(||T_k(\Sigma_{0})_{YX}- T_k(\Sigma)_{YX} ||^2 I[\lambda_{\min}\{\tilde{M}(\Sigma_{XX})\}<M_1/2]\mid\bbZ_n)\}\\
		&&+ \frac{16}{M_1^2} E_{\Sigma_0}\{E^{\pi^i}(||T_k(\Sigma_{0})_{YX}- T_k(\Sigma)_{YX} ||^2\mid\bbZ_n)\} \\
		&\le& \frac{4}{\epsilon_n^2} E_{\Sigma_0}\{E^{\pi^i}(||T_k(\Sigma_{0})- T_k(\Sigma) ||^4\mid\bbZ_n)\}^{1/2} E_{\Sigma_0}[P^{\pi^i}\{\lambda_{\min}(\Sigma)<M_1/2\mid\bbZ_n\}]^{1/2}\\
		&&+ \frac{16}{M_1^2} E_{\Sigma_0}\{E^{\pi^i}(||T_k(\Sigma_{0})_{YX}- T_k(\Sigma)_{YX} ||^2\mid\bbZ_n)\} \\
		&\le&  C_2\Big(\frac{1}{\epsilon_n^2} p^{1/2}5^{k/2} \exp(-\lambda_2n) + \frac{k+\log p}{n} + k^{-2\alpha}\Big),
		\eea    
		for all sufficiently large $n$. The third inequality holds since $\lambda_{\min}\{\tilde{M}(\Sigma_{XX})\} \ge \lambda_{\min}(\Sigma)$. The last inequality holds by Lemma \ref{lemma:upper2}, \ref{lemma:evmin} and \ref{lemma:TPPPupper}.
		For the upper bound of \eqref{formula:PPPupper3}, we have
		\bean
		&&E_{\Sigma_0}\{E^{\pi^i}( \Lambda^{(0)}\{T_k(\Sigma_{0,XX});2\lfloor ak\log k \rfloor\}-\Lambda^{(\epsilon_n)}\{ T_k(\Sigma)_{XX};2\lfloor ak\log k \rfloor\}||^2\mid\bbZ_n)\}\nonumber\\
		&=&E_{\Sigma_0}\{E^{\pi^i}(||T_k\{\tilde{M}(\Sigma_{0,XX}) \}^{-1} -T_k^{(\epsilon_n)}\{\tilde{M}(\Sigma_{XX})\}^{-1}
		||^2\mid\bbZ_n)\} \nonumber\\
		&\le& 2E_{\Sigma_0}\{E^{\pi^i}(||\tilde{M}(\Sigma_{0,XX} )^{-1} - T_k^{(\epsilon_n)}\{\tilde{M}(\Sigma_{XX})\}^{-1}||^2\mid\bbZ_n)\} \label{formula:thmupperPPP1}\\
		&&+ 2||\tilde{M}(\Sigma_{0,XX} )^{-1} -T_k\{\tilde{M}(\Sigma_{0,XX}) \}^{-1}||^2\label{formula:thmupperPPP2}
		\eean
		where the first equality holds by the definition of $\Lambda^{(0)}$ and $\Lambda^{(\epsilon_n)}$. Since $\tilde{M}(\Sigma_{0,XX})\in\calF_{2\lfloor ak\log k\rfloor,\alpha}(M,M_0,M_1)$, Lemma \ref{main-lemma:l1bandinv} gives that there exists some positive constant $C_3$ depending only on $M$, $M_0$ and $M_1$ such that 
		\bea
		\eqref{formula:thmupperPPP2} \le C_3 (\lfloor k/2\rfloor)^{-2\alpha} 
		\eea
		for all sufficiently large $k$. 
		Lemma \ref{lemma:invfactor} gives that there exist some positive constants $C_4$ and $\lambda_3$ depending only on $M$, $M_0$, $M_1$ and $\alpha$ such that 
		\bea
		\eqref{formula:thmupperPPP1}
		&\le& 	C_4 \Big\{  \frac{k+\log p}{n} + k^{-2\alpha} + \epsilon_n^2 +\frac{p^{1/2} 5^{k/2} \exp(-\lambda_3 n)}{\epsilon_n^2}  \Big\},
		\eea
		for all sufficiently large $n$. Thus, there exists some positive constant $C_5$ depending only on $M$, $M_0$, $M_1$ and $\alpha$ such that 
		\bea
		\eqref{formula:PPPupper3}\le C_5 \Big\{ \frac{k+\log p}{n} + k^{-2\alpha} + \epsilon_n^2 +\frac{p^{1/2} 5^{k/2} \exp(-\lambda_3 n)}{\epsilon_n^2} \Big\},
		\eea
		for all sufficiently large $n$ and $k$. 
		Collecting the upper bounds of \eqref{formula:PPPupper1}, \eqref{formula:PPPupper2} and \eqref{formula:PPPupper3}, we complete the proof.

	\end{proof}
	
	\section{Proofs of lemmas in Appendix A}
	
	In this section, we prove Lemmas \ref{main-lemma:l1bandinv}-\ref{main-lemma:ineqTV} which are presented in Appendix A. For the proofs we present Lemmas \ref{lemma:taper} and \ref{lemma:momentscov}.
	
	\begin{lemma}\label{lemma:taper}
		Let $p$ and $k$ be positive integers with $k \le p$.
		For an arbitrary covariance matrix $\Sigma\in\calC_p$, 
		\bea
		||T_k(\Sigma) ||_r \le 3 \max_{1 \le l\le p} || M_l^{(k)}(\Sigma)||_r,
		\eea
		where $||\cdot||_r$ is the matrix $r$-norm.
	\end{lemma}
	
	\begin{proof}
		
		\cite{cai2010optimal} in Lemma 1 shows
		\bea
		T_k(\Sigma) = (k/2)^{-1} \{ S^{*(k)}(\Sigma) -S^{*(k/2)}(\Sigma) \},
		\eea
		where $S^{*(k)}(\Sigma) = \sum_{l=1-k}^p M_l^{*(k)}(\Sigma)$.
		Thus, we have
		\bean
		||T_k(\Sigma)||_r \le (2/k) (|| S^{*(k)}(\Sigma) ||_r + ||S^{*(k/2)}(\Sigma)||_r).\label{formula:taper1}
		\eean
		Note that 
		\bea
		S^{*(k)}(\Sigma)  = \sum_{l=1}^k \sum_{-1\le j \le \lfloor p/k\rfloor} M_{jk+l}^{*(k)}(\Sigma) .
		\eea
		We have 
		\bea
		\Big|\Big|\sum_{-1\le j \le \lfloor p/k\rfloor} M_{jk+l}^{*(k)}(\Sigma) \Big|\Big|_r = \max_{-1\le j \le \lfloor p/k\rfloor}  ||M_{jk+l}^{*(k)}(\Sigma) ||_r,
		\eea
		for $l\in \{1,2,\ldots,k\}$, since the sum of matrices is a block diagonal matrix. 
		Thus, we get    	
		\bean
		||S^{*(k)}(\Sigma)||_r &\le& \sum_{l=1}^k \Big|\Big|\sum_{-1\le j \le \lfloor p/k\rfloor} M_{jk+l}^{*(k)}(\Sigma) \Big|\Big|_r \nonumber\\
		&\le& k\max_{1\le l \le k} \Big|\Big|\sum_{-1\le j \le \lfloor p/k\rfloor} M_{jk+l}^{*(k)}(\Sigma) \Big|\Big|_r \nonumber\\
		&\le& k\max_{1-k\le l \le p} ||M_l^{*(k)}(\Sigma)||_r,\nonumber\\
		&=& k\max_{1\le l \le p} ||M_l^{*(k)}(\Sigma)||_r,\label{formula:taper2}
		\eean
		and 
		\bea
		||S^{*(k/2)} (\Sigma)||_r&\le& (k/2)\max_{1\le l \le p}||M_l^{*(k/2)}  (\Sigma)||_r.
		\eea
		Since $
		||M_l^{*(k/2)} (\Sigma)||_r \le ||M_{l}^{*(k)} (\Sigma)||_r$ for $l\in \{1,2,\ldots,p\}$, we have 
		\bean
		||S^{*(k/2)} (\Sigma)||_r &\le&(k/2)\max_{1-k\le l \le p}||M_l^{*(k)}(\Sigma)||_r .\label{formula:taper3}
		\eean
		Collecting \eqref{formula:taper1}, \eqref{formula:taper2} and  \eqref{formula:taper3}, we get
		\bea
		||T_k(\Sigma)||_r \le 3\max_{1\le l \le p }||M_l^{(k)} (\Sigma) ||_r. ~	 
		\eea
	\end{proof}

	\begin{lemma}\label{lemma:momentscov}
		Let $k$ and $r$ be positive integers, and $\tau$ and $M_0$ be positive real numbers with $\tau\ge k-1$.
		Suppose $\Sigma_0 \in \calC_k$ and $T \sim W_k(\Sigma_0,\tau)$. If $\lambda_{\max}(\Sigma_0) \le M_0$, then there exists some positive constant $C$ depending only on $M_0$ and $r$ such that
		\bea
		E(||T/\tau-\Sigma_0||_2^r) \le C\frac{5^k}{\tau^{r/2}}.
		\eea
		
	\end{lemma}
	
	\begin{proof}
		
		Let $\Omega_{\tau}^{(k)}= \Sigma_0^{-1/2}T\Sigma_0^{-1/2}/\tau$. We have 
		\bea
		E(||T/\tau-\Sigma_0||_2^r) &\le&
		||\Sigma_0||_2^r E(||\Omega_{\tau}^{(k)}- I_k||_2^r).
		\eea
		Since $||\Sigma_0||_2^r$ is bounded by $M_0^r$, it suffices to show the upper bound of $E(||\Omega_{\tau}^{(k)}- I_k||_2^r) $.
		By a property of Wishart distribution, we have
		\bea
		\Omega_{\tau}^{(k)}=\Sigma_0^{-1/2}T\Sigma_0^{-1/2}/\tau \sim W_k(I_k/\tau,\tau),
		\eea
		and 
		\bea
		E(||\Omega_{\tau}^{(k)} - I_k||^r ) &\le& \int_0^1 x^{r-1} P(||\Omega_{\tau}^{(k)} - I_k||>x) dx + \int_1^\infty x^{r-1}P(||\Omega_{\tau}^{(k)} - I_k||>x) dx\\
		&\le& 2\times 5^k \Big\{\int_0^1 x^{r-1}\exp(-\tau x^2/2^7)dx + \int_1^\infty x^{r-1}\exp(-\tau x/2^7)dx\Big\},
		\eea
		where the last inequality is satisfied by Lemma 4.4 in \cite{lee2020post}.
		Integration by substitution and the definition of Gamma function give the inequalities 
		\bea
		\int_0^1 x^{r-1}\exp(-\tau x^2/2^7)dx &\le& 
		\frac{2^{7r/2-1}}{\tau ^{r/2}}\int_0^\infty t^{r/2-1} \exp(-t) dt \\
		&=& \frac{2^{7r/2-1}}{\tau^{r/2}} \Gamma(r/2)\\
		\int_1^\infty x^{r-1}\exp(-\tau x/2^7)dx &\le& \Big(\frac{2^7}{\tau }\Big)^r\int_0^\infty t^{r-1}\exp(-t)dx\\
		&=& \Big(\frac{2^7}{\tau }\Big)^r\Gamma(r).
		\eea
		Thus, 
		\bea
		E(||\Omega_{\tau }^{(k)} - I_k||^r ) &\le& C_1\frac{5^k}{\tau^{r/2}}, 
		\eea
		for some positive constant $C_1$ depending only on $r$. 	 
		
	\end{proof}
	
	Next, we prove Lemmas \ref{main-lemma:l1bandinv}-\ref{main-lemma:ineqTV}. 
	
	\begin{proof}[Proof of Lemma \ref{main-lemma:l1bandinv}]
		By Thorem 2.3.4 of \cite{golub2013matrix}, we have
		\bean
		||\Sigma_0^{-1} - B_k(\Sigma_0)^{-1} ||_1 &\le& 
		\frac{||B_k(\Sigma_0)^{-1} ||_1^2 || \Sigma_0-B_k(\Sigma_0)||_1}{1-||B_k(\Sigma_0)^{-1} \{\Sigma_0-B_k(\Sigma_0)\}  ||_1},\label{formula:l1bandinv1}
		\eean
		when $||B_k(\Sigma_0)^{-1} \{\Sigma_0-B_k(\Sigma_0)\}  ||_1<1$. This inequality condition holds for all sufficiently large $k$, since $||B_k(\Sigma_0)^{-1}||_1$ is bounded above by a constant and $||\Sigma_0-B_k(\Sigma_0)||_1\le Mk^{-\alpha}$.
		We show the boundedness of $||B_k(\Sigma_0)^{-1}||_1$ below.
		Note that $B_k(\Sigma_0) \in \calC_p$ for all sufficiently large $k$ since $\lambda_{\min}\{B_k(\Sigma_0)\} \ge \lambda_{\min}(\Sigma_0)- ||B_k(\Sigma_0)- \Sigma_0||$ \citep[Lemma 4.14]{lee2020post}.
		Thus, Theorem 2.4 of \cite{demko1984decay} gives 
		\bean\label{formula:l1bandinv2}
		||B_k(\Sigma_0)^{-1} ||_1 &\le& \max_j \sum_{i=1}^p | w_{ij} | \nonumber\\
		&\le& C_1\sup_j \sum_{i=1}^p (q_1^{2/k})^{|i-j|} \nonumber\\
		&\le&  \frac{2C_1}{1-q_1^{2/k}},
		\eean
		where $B_k(\Sigma_0)=(w_{ij})$, $C_1 = \{||B_k(\Sigma_0)^{-1}|| \vee (1+\kappa_1^{1/2})^2\}/(2||B_k(\Sigma_0)||)$, $q_1=(\kappa_1^{1/2}-1)/(\kappa_1^{1/2}+1)$ and $\kappa_1$ is the spectral condition number of $B_k(\Sigma_0)$. 
		To show that \eqref{formula:l1bandinv2} is bounded, it suffices to show that $||B_k(\Sigma_0)^{-1}||_2$ is bounded.
		By Lemma 4.14 in \cite{lee2020post}, we have
		\bean\label{formula:l1bandinv3}
		||B_k(\Sigma_0)^{-1}||_2 &=& \lambda_{\min}\{B_k(\Sigma_0)\}^{-1} \nonumber\\
		&\le& \frac{1}{\lambda_{\min}(\Sigma_0) -  ||B_k(\Sigma_0)-\Sigma_0||}\nonumber\\
		&\le& \frac{1}{\lambda_{\min}(\Sigma_0) -  Mk^{-\alpha}},
		\eean
		which is bounded for all sufficiently large $k$. Thus, for all sufficiently large $k$, $||B_k(\Sigma_0)^{-1} \{\Sigma_0-B_k(\Sigma_0)\}  ||_1<1$ and the inequality \eqref{formula:l1bandinv1} is satisfied. 
		By combining \eqref{formula:l1bandinv1} and \eqref{formula:l1bandinv2}, we have that there exists some positive constant $C_2$ depending only on $M$, $M_0$ and $M_1$ such that
		\bea
		||\Sigma_0^{-1} - B_k(\Sigma_0)^{-1} ||_1 &\le&  \Big( \frac{2C_1 }{1-q_1^{2/k}} \Big)^2 \frac{ Mk^{-\alpha} }{ 1-2C_1 Mk^{-\alpha} / (1-q_1^{2/k}) }\\
		&\le& C_2M k^{-\alpha} ,
		\eea
		for all sufficiently large $k$.
		
		Next, we consider the upper bound of
		$||\Sigma_0^{-1} - T_k(\Sigma_0)^{-1} ||_1 $.
		Note that $T_k(\Sigma_0)$ is a $k$-band matrix and 
		\bean\label{formula:l1bandinv4}
		||T_k(\Sigma_0)-\Sigma_0||_1 &\le& ||B_{\lfloor k/2\rfloor }(\Sigma_0)-\Sigma_0 ||_1\nonumber\\
		&\le& M(\lfloor k/2\rfloor )^{-\alpha}.
		\eean
		In the same way as deriving the inequalties \eqref{formula:l1bandinv2} and \eqref{formula:l1bandinv3}, we have
		\bean
		||T_k(\Sigma_0)^{-1}||_1 &\le& \frac{2C_3}{1-q_2^{2/k}}\label{formula:l1bandinv5}\\
		||T_k(\Sigma_0)^{-1}||_2 &\le& \frac{1}{\lambda_{\min}(\Sigma_0) - M(\lfloor k/2\rfloor )^{-\alpha}},\nonumber
		\eean
		where $C_3 = ||T_k(\Sigma_0)^{-1}|| \vee (1+\kappa_2^{1/2})^2/(2||T_k(\Sigma_0)||)$, $q_2=(\kappa_2^{1/2}-1)/(\kappa_2^{1/2}+1)$ and $\kappa_2$ is the spectral condition number of $T_k(\Sigma_0)$. 
		Thus, we have $||T_k(\Sigma_0)^{-1}(\Sigma_{0}-T_k(\Sigma_{0}))||_1<1$ for all sufficiently large $k$, and apply Theorem 2.3.4 of \cite{golub2013matrix}.
		By Theorem 2.3.4 of \cite{golub2013matrix}, there exists some positive constant $C_4$ depending only on $M$, $M_0$ and $M_1$ such that 
		\bea
		||\Sigma_0^{-1} - T_k(\Sigma_0)^{-1} ||_1  &\le& 	\frac{||T_k(\Sigma_0)^{-1} ||_1^2 || \Sigma_0-T_k(\Sigma_0)||_1}{1-||T_k(\Sigma_0)^{-1} (\Sigma_0-T_k(\Sigma_0))  ||_1}\\
		&\le& C_4  (\lfloor k/2\rfloor )^{-\alpha},
		\eea
		for all sufficiently large $k$. The last inequality holds by the inequalities \eqref{formula:l1bandinv4} and \eqref{formula:l1bandinv5}.

	\end{proof}

	\begin{proof}[Proof of Lemma \ref{main-lemma:evminfreq}]
		Two inequalities in Lemma 4.14 of \cite{lee2020post},
		\bea
		\lambda_{\min}\{T_k(S_n)\} &\ge& \lambda_{\min}\{T_k(\Sigma_0)\} -||T_k(S_n)-T_k(\Sigma_0)||\\
		\lambda_{\min}\{T_k(\Sigma_0)\} &\ge& \lambda_{\min}(\Sigma_0) -||\Sigma_0-T_k(\Sigma_0)||,
		\eea
		imply
		\bean
		I[\lambda_{\min}\{T_k(S_n)\}\le c]
		&\le&I\{\lambda_{\min}(\Sigma_0)- ||T_k(\Sigma_0)-\Sigma_0||- ||T_k(S_n)-T_k(\Sigma_0)|| \le c\}\nonumber\\
		&\le& I( ||T_k(S_n)-T_k(\Sigma_0)|| \ge c_k), \label{formula:evminfreq1}
		\eean
		where $c_k = \lambda_{\min}(\Sigma_0) -M(\lfloor k/2\rfloor )^{-\alpha}  -c $. The last inequality holds by the inequality \eqref{formula:l1bandinv4}.
		Since $c\le \lambda_{\min}(\Sigma_0)/2$ and $\lfloor k/2 \rfloor> \{4M/\lambda_{\min}(\Sigma_0)\}^{1/\alpha}$, we have $c_k \ge \lambda_{\min}(\Sigma_0)/4$.
		By Lemma \ref{lemma:taper}, we have   
		\bea
		P( ||T_k(S_n)-T_k(\Sigma_0)|| \ge c_k) &\le&
		P(\max_{1\le l\le p} ||M_l^{(k)}(S_n-\Sigma_0) ||  \ge c_k/3)\\
		&\le& P(||\Sigma_0|| \max_{1\le l\le p} ||\Omega_{n,l}^{*(k)}-I_k ||  \ge c_k/3)\\
		&\le& P(M_0 \max_{1\le l\le p} ||\Omega_{n,l}^{*(k)}-I_k ||  \ge c_k/3)\\
		&\le&  p \max_{1\le l\le p} P\{||\Omega_{n,l}^{*(k)}-I_k ||  \ge c_k/(3M_0)\},
		\eea
		where $\Omega_{n,l}^{*(k)}  = M_l^{(k)}(\Sigma_0)^{-1/2}M_l^{(k)}(S_n)M_l^{(k)}(\Sigma_0)^{-1/2}$ of which distribution is $W_k(I_k/n,n)$.
		Since $c_k/(3M_0)<1$, 
		Lemma 4.4 in \cite{lee2020post} gives
		$$
		P \{||\Omega_{n,l}^{*(k)}-I_k ||\ge c_k/(3M_0)\}\le 2\times 5^k \exp[-n \{c_k/(3M_0)\}^2/2^7 ]
		$$	
		Thus, combining this inequality with \eqref{formula:evminfreq1}, we get
		\bea
		E(I[\lambda_{\min}\{T_k(S_n)\}\le c] ) &\le& E\{ I( ||T_k(S_n)-T_k(\Sigma_0)|| \ge c_k) \} \\
		&\le& p \max_{1\le l\le p} P\{||\Omega_{n,l}^{*(k)}-I_k ||  \ge c_k/(3M_0)\} \\
		&\le&  2 p 5^k \exp(-\lambda n),
		\eea
		for some positive constant $\lambda$ depending only on $M_0$ and $M_1$.

	\end{proof}

	\begin{proof}[Proof of Lemma \ref{main-lemma:upper1}]
		Note, by Lemmas \ref{lemma:taper} and \ref{lemma:pdadjust},  
		\bea
		E(||T_k^{(\epsilon_n)}(S_n)-\Sigma_0||^4) &\le& 
		2^7 E(||T_k(S_n)-\Sigma_0||^4) + 4^3 \epsilon_n^4\\
		&\le& 2^{10} E(||T_k(S_n)-T_k(\Sigma_0)||^4)  +2^{10} ||\Sigma_0-T_k(\Sigma_0)||^4 + 4^3 \epsilon_n^4\\
		&\le& 2^{10}3^4 E(\max_{1\le l\le p}||M_l^{(k)}(S_n-\Sigma_0)||^4)  +2^{10} ||\Sigma_0-T_k(\Sigma_0)||^4 + 4^3 \epsilon_n^4.
		\eea
		It suffices to show $E(\max_{1\le l\le p}||M_l^{(k)}(S_n-\Sigma_0)||^4)$ is bounded above.
		There exist some positive constants $\rho_1$ and $C_1$ depending only on $M_0$ such that
		\bea
		&&E(\max_{1\le l\le p}||M_l^{(k)}(S_n-\Sigma_0)||^4)\\
		&\le &  x^4 + E\{\max_{1\le l\le p}||M_l^{(k)}(S_n-\Sigma_0)||^4 I(\max_{1\le l\le p}||M_l^{(k)}(S_n-\Sigma_0)||>x)\}\\
		&\le & x^4 + p^{1/2}\max_{1\le l\le p}E(||M_l^{(k)}(S_n-\Sigma_0)||^8)^{1/2} P(\max_{1\le l\le p}||M_l^{(k)}(S_n-\Sigma_0)||>x)^{1/2}\\
		&\le & x^4 + C_1 p^{1/2} \frac{5^{k/2} }{n^2} \{2p 5^k \exp(-nx^2 \rho_1)\}^{1/2},
		\eea
		for all $0<x<\rho_1$. The last inequality is satisfied by Lemma \ref{lemma:momentscov} and Lemma 3 of \cite{cai2010optimal}. 
		We set $x= \{2 (\log p +k)\log 5/(n\rho_1)\}^{1/2}$, which can be smaller than an arbitrary positive constant for all sufficiently large $n$. Then, we have
		\bea
		E(\max_{1\le l\le p}||M_l^{(k)}(S_n-\Sigma_0)||^4) &\le&
		\Big\{\frac{2(\log p +k )\log 5 }{n\rho_1}\Big\}^2 + \frac{\sqrt{2}C_1 p 5^k}{n^2}  \exp\{- (\log p+k)\log 5\}\\
		&=& \Big\{\frac{2 (\log p +k )\log 5}{n\rho_1}\Big\}^2  + \frac{\sqrt{2}C_1}{n^2} \exp(-(\log 5 - 1)\log p ),
		\eea
		which is bounded by some positive constant depending only on $M_0$ for all sufficiently large $n$.	 
		
	\end{proof}

	\begin{proof}[Proof of Lemma \ref{main-lemma:upperblockinv}]
		
		Let $c=\lambda_{\min}(\Sigma_0)/2$. Since $\lambda_{\min}\{T_k^{(\epsilon_n)}(S_n)_{11}\}\ge \lambda_{\min}\{T_k^{(\epsilon_n)}(S_n)\}\ge \epsilon_n$, we have 
		\bean
		&&E(|| T_k^{(\epsilon_n)}(S_n)_{11}^{-1} -\Sigma_{0,11}^{-1} ||^2)\nonumber\\
		&\le&
		||\Sigma_{0,11}^{-1}||^2 E(|| T_k^{(\epsilon_n)}(S_n)_{11}^{-1} ||^2 
		|| T_k^{(\epsilon_n)}(S_n)_{11} -\Sigma_{0,11} ||^2 )\nonumber\\
		&\le&  \frac{||\Sigma_{0,11}^{-1}||^2 }{\epsilon_n^2} E(||T_k^{(\epsilon_n)}(S_n)_{11} -\Sigma_{0,11} ||^2 I[\lambda_{\min}\{T_k(S_n)\} \le c]) \label{formula:upperblockinv1}\\
		&&+ \frac{||\Sigma_{0,11}^{-1}||^2 }{c^2}E( ||T_k^{(\epsilon_n)}(S_n)_{11}  -\Sigma_{0,11} ||^2).\label{formula:upperblockinv2}
		\eean
		By Lemmas \ref{main-lemma:evminfreq} and \ref{main-lemma:upper1}, 
		there exist some positive constants $C_1$ and $\lambda_1$ depending only on $M_0$ and $M_1$ such that
		\bean
		\eqref{formula:upperblockinv1}
		&\le& \frac{1}{M_1^2 \epsilon_n^2}E(||T_k^{(\epsilon_n)}(S_n)_{11} -\Sigma_{0,11} ||^4)^{1/2} P[\lambda_{\min}\{T_k(S_n)\} \le c]^{1/2} \nonumber\\
		&\le&  \frac{C_1}{M_1^2 \epsilon_n^2}p^{1/2} 5^{k/2} \exp(-\lambda_1 n) ,\label{formula:inv2}
		\eean
		for all sufficiently large $n$.
		We also have
		\bean
		\eqref{formula:upperblockinv2}
		&\le& 	\frac{1}{M_1^2 c^2}	E( ||T_k^{(\epsilon_n)}(S_n)  -\Sigma_{0} ||^2) \nonumber\\
		&\le& \frac{2^3}{M_1^2 c^2}E( ||T_k(S_n)  -\Sigma_{0} ||^2) + \frac{4}{M_1^2 c^2}\epsilon_n^2 \nonumber\\
		&\le& C_2 \Big\{  \frac{k + \log p}{n} + k^{-2\alpha} + \epsilon_n^2\Big\},\label{formula:inv3}
		\eean
		for some positive constant $C_2$ depending only on $M$, $M_0$, $M_1$ and $\alpha$. The second and last inequalities hold by Lemma \ref{lemma:pdadjust} and Theorem 2 of \cite{cai2010optimal}, respectively.
		Collecting \eqref{formula:inv2} and \eqref{formula:inv3}, we complete the proof.	 
	\end{proof}
	
	\begin{proof}[Proof of Lemma \ref{main-lemma:ineqTV}]
		Let $c=||\Sigma^{-1}||_2 \wedge ||\Sigma'^{-1}||_2  $ and $D = \Sigma' - \Sigma$.
		Without loss of generality we suppose $|| \Sigma^{-1}||_2 = c$.
		The Pinsker inequality gives 
		\bean
		|| P_\Sigma - P_{\Sigma'} ||_1^2 &\le& 2KL(P_{\Sigma'} \mid P_{\Sigma}) \nonumber\\
		&=& n  \{tr(\Sigma' \Sigma^{-1}) - \log \det (\Sigma' \Sigma^{-1}) - p\},\nonumber\\
		&=& n \{tr( D\Sigma^{-1})  - \log \det (I_p + D\Sigma^{-1}) \},\label{formula:ineqTV1}
		\eean
		where $KL(\cdot\mid \cdot)$ is the Kullback-Leibler divergence. Before showing the upper bound of \eqref{formula:ineqTV1}, we show
		\bean
		\log (1+x) \ge x- x^2, ~ \text{for } x>-1/2. \label{formula:ineqTV2}
		\eean
		Let $g(x) = \log (1+x)-x + x^2$. Note  
		\bea
		g(0) &=& 0\\
		g'(x) &<&0,  ~ \text{if } -1/2 < x <0\\
		g'(x) &>&0,  ~ \text{if }  x>0,
		\eea
		where $g'(x)$ is the derivative of $g(x)$.
		Thus, inequality \eqref{formula:ineqTV2} holds.
		
		Now we give the upper bound of \eqref{formula:ineqTV1}. 
		Note that $D\Sigma^{-1}$ is a similar matrix of $\Sigma^{-1/2}D\Sigma^{-1/2}$ and
		the set of eigenvalues of  $D\Sigma^{-1}$ coincides with that of $\Sigma^{-1/2}D\Sigma^{-1/2}$.
		Let $\{\lambda_1,\lambda_2,\ldots,\lambda_p\}$ be the set of eigenvalues of $D\Sigma^{-1}$. 
		If $\lambda_i >-1/2$ for all $i\in\{1,\ldots,p\}$, then inequality \eqref{formula:ineqTV2} gives
		\bea
		-\log \det (I_p + D\Sigma^{-1})  &=&  -\sum_{i=1}^p \log (1 + \lambda_i)\\
		&\le& -\sum_{i=1}^p\lambda_i + \sum_{i=1}^p \lambda_i^2\\
		&=& -tr(D\Sigma^{-1})+ \sum_{i=1}^p \lambda_i^2.
		\eea
		By applying this inequality to \eqref{formula:ineqTV1}, we have
		\bea
		|| P_\Sigma - P_{\Sigma'} ||_1^2 \le n \sum_{i=1}^p \lambda_i^2.
		\eea
		It suffices to show $\max_{i=1,\ldots,p} |\lambda_i| <1/2$ and $\sum_{i=1}^p \lambda_i^2 \le   c^2||D||_F^2$.
		Since $||D||_2< 1/(2c)$,
		\bea
		\max_{i=1,\ldots,p}|\lambda_i| &\le& ||\Sigma^{-1/2}D\Sigma^{-1/2}||_2 \\
		&\le&  ||D||_2 ~||\Sigma^{-1}||_2 \\
		&<&  1/2. 
		\eea
		Next, we show the upper bound of $\sum_{i=1}^p \lambda_i^2$. we have $\Sigma^{-1} = UVU^T$ for some orthogonal matrix $U$ and some diagnoal matrix $V$, and
		\bea
		\sum_{i=1}^p \lambda_i^2  &=& || \Sigma^{-1/2}D\Sigma^{-1/2}||_F^2\\
		&=&  || V U^T DU V||_F^2\\
		&\le&  ||V||^2 ~ || U^T DU||_F^2 \\
		&=& ||\Sigma^{-1}||^2 ~ ||D ||_F^2.
		\eea
		The second and third equalities hold since $U$ is an orthogonal matrix. The first inequality holds since $V$ is a diagonal matrix.	  
		
	\end{proof}

	\se{Proofs of Lemmas in Section S4}\label{sec:proofmis}
	
	In this section, we prove Lemmas \ref{lemma:upper2}-\ref{lemma:invfactor} which are presented in Section \ref{sec:thm6}. For these proofs we present Lemmas \ref{lemma:IWmomentbound} and \ref{lemma:IWconcent}.
	
	\begin{lemma}\label{lemma:IWmomentbound}
		Let $n$, $k$ and $r$ be positive integers and $M_0$ be a positive real number.
		Suppose $\Sigma_0\in\calC_k$ and let the prior distribution $\pi$ on $\calC_k$ be $IW(A_n,\tau_n)$ for $A_n\in\calC_k$ and $\tau_n>2k$.
		If $(\tau_n-2k)\vee||A_n||\vee k = o(n)$ and $||\Sigma_0||\le M_0$, then there exists some positive constant $C$ depending only on $M_0$ and $r$ such that
		\bea
		E_{\Sigma_0}\{E^{\pi}(||\Sigma-\Sigma_0||^r\mid \bbZ_n)\} \le C k^{r} \frac{5^{3k/2}}{n^{r/2}} ,
		\eea
		for all sufficiently large $n$.
	\end{lemma}
	
	\begin{proof}
		Let $\hat{\Sigma} = (nS_n+A_n)/(n+\tau_n-k-1)$. We have  
		\bean
		E_{\Sigma_0}\{E^{\pi}(||\Sigma-\Sigma_0||^r\mid \bbZ_n)\} &\le& 2^{r-1} E_{\Sigma_0}\{E^{\pi}(||\Sigma-\hat{\Sigma}||^r\mid \bbZ_n)\} \label{formula:IWmomentbound1}\\ &&+2^{r-1}E_{\Sigma_0}(||\hat{\Sigma}-\Sigma_0||^r).\label{formula:IWmomentbound2}
		\eean
		For the upper bound of \eqref{formula:IWmomentbound2}, we have
		\bea
		E_{\Sigma_0}(||\hat{\Sigma}-\Sigma_0||^r ) &\le& 
		2^{r-1} E_{\Sigma_0}\Big(\Big|\Big|\frac{n}{n+\tau_n-k-1}(S_n-\Sigma_0)\Big|\Big|^r\Big) 
		+
		2^{r-1} \Big|\Big|\frac{A_n -(\tau_n-k-1)\Sigma_0}{n+\tau_n-k-1}\Big|\Big|^r.
		\eea
		Since $(\tau_n -k -1)/n = O(k/n)$ and $||A_n||=o(n)$, there exists some positive constant $C_1$ depending only on $M_0$ and $r$ such that
		\bea
		\Big|\Big|\frac{A_n -(\tau_n-k-1)\Sigma_0}{n+\tau_n-k-1}\Big|\Big|^r &\le& C_1\Big(\frac{k}{n}\Big)^{r},
		\eea
		for all sufficiently large $n$.
		Lemma \ref{lemma:momentscov} gives  
		\bea
		E_{\Sigma_0}\Big(\Big|\Big|\frac{n}{n+\tau_n-k-1}(S_n-\Sigma_0)\Big|\Big|^r\Big)  \le C_2 \frac{5^k}{n^{r/2}},
		\eea
		for some positive constant $C_2$ depending only on $M_0$ and $r$. Thus, we have
		\bean\label{formula:IWmomentbound8}
		\eqref{formula:IWmomentbound2} \le 4^{r-1} C_1\Big(\frac{k}{n}\Big)^r + 4^{r-1} C_2 \frac{5^k}{n^{r/2}}.
		\eean
		
		Next, we show the upper bound of \eqref{formula:IWmomentbound1}. We have
		\bean \label{formula:IWmomentbound3}
		&&E_{\Sigma_0}\{E^{\pi}(||\Sigma-\hat{\Sigma}||^r\mid \bbZ_n)\} \nonumber\\
		&=& E_{\Sigma_0}\{E^{\pi}(||   \hat{\Sigma} ^{1/2}  (\Omega_{n^*}^{(k)} )^{-1}(I_k - \Omega_{n^*}^{(k)} )\hat{\Sigma} ^{1/2}   ||^r\mid \bbZ_n)\} \nonumber\\
		&\le& 
		E_{\Sigma_0}\{||\hat{\Sigma} ||^{r} E^{\pi}(||(\Omega_{n^*}^{(k)})^{-1}||^r  ||I_k-\Omega_{n^*}^{(k)}||^r\mid \bbZ_n)\} \nonumber\\
		&\le& 
		E_{\Sigma_0}\{||\hat{\Sigma} ||^{r} E^{\pi}(||(\Omega_{n^*}^{(k)})^{-1}||^{2r}\mid \bbZ_n)^{1/2}  E^{\pi}(||I_k-\Omega_{n^*}^{(k)}||^{2r}\mid \bbZ_n)^{1/2}\},
		\eean
		where $\Omega_{n^*}^{(k)} = \hat{\Sigma}^{1/2}\Sigma^{-1}\hat{\Sigma}^{1/2}$. 
		Note that $[\Sigma\mid \bbZ_n] \sim IW_k (A_n+nS_n, n+\tau_n)$. 
		Since $[\Sigma^{-1}\mid \bbZ_n] \sim W_k\{(A_n+nS_n)^{-1}, n+\tau_n -k-1\}$ \citep{press2012applied}, we have
		$(n+\tau_n-k-1) \Omega_{n^*}^{(k)}\sim W_k(I_k,n+\tau_n-k-1)$.
		Lemma \ref{lemma:momentscov} gives
		\bean\label{formula:IWmomentbound4}
		E^\pi(||I_k-\Omega_{n^*}^{(k)}||^{2r}\mid S_n)^{1/2} &\le& C_3\frac{5^{k/2}}{(n+\tau_n-k-1)^{r/2}}\nonumber\\
		&\le&C_3\frac{5^{k/2}}{n^{r/2}} ,
		\eean
		for some positive constant $C_3$ depending only on $M_0$ and $r$. 
		Since $[(\Omega_{n^*}^{(k)})^{-1} \mid \bbZ_n]\sim IW_k\{(n+\tau_n-k-1)I_k,n+\tau_n\}$, we have
		\bean \label{formula:IWmomentbound5}
		E^{\pi}(||(\Omega_{n^*}^{(k)})^{-1}||^{2r}\mid \bbZ_n) &\le& 
		E^{\pi}([\sum_{i=1}^k \{(\Omega_{n^*}^{(k)})^{-1}\}_{ii}]^{2r} \mid \bbZ_n)\nonumber\\
		&\le& k^{2r-1} \sum_{i=1}^k  E^{\pi}[\{(\Omega_{n^*}^{(k)})^{-1}\}_{ii}^{2r} \mid \bbZ_n]\nonumber\\
		&\le& k^{2r}(n+\tau_n-k-1)^{2r}\frac{1}{2^{2r}} \frac{\Gamma\{(n+\tau_n-2k)/2-2r\} }{\Gamma\{(n+\tau_n-2k)/2\} }\nonumber\\
		&=& k^{2r}(n+\tau_n-k-1)^{2r}\frac{1}{2^{2r}} \prod_{i=1}^{2r} \{(n+\tau_n-2k)/2-i\}^{-1} \nonumber\\
		&\le& \Big(\frac{k}{2}\Big)^{2r} \Big\{\frac{n+\tau_n-k-1}{(n+\tau_n-2k)/2-2r}\Big\}^{2r},
		\eean
		where $\{(\Omega_{n^*}^{(k)})^{-1}\}_{ii}$ is the $i$th diagonal element of $(\Omega_{n^*}^{(k)})^{-1}$. The third inequality holds by Lemma 4.9 of \cite{lee2020post}. By collecting \eqref{formula:IWmomentbound3}, \eqref{formula:IWmomentbound4} and \eqref{formula:IWmomentbound5}, we have that there exists  some positive constant $C_4$ depending only on $M_0$ and $r$ such that
		\bean
		\eqref{formula:IWmomentbound1}
		&\le& C_4k^{r}\frac{5^{k/2}}{n^{r/2}}E_{\Sigma_0}(||\hat{\Sigma} ||^{r}),\label{formula:IWmomentbound6}
		\eean
		for all sufficiently large $n$.
		
		For the upper bound of $E_{\Sigma_0}(||\hat{\Sigma} ||^{r})$,
		there exist some positive constants $C_5$ and $C_6$ depending only on $M_0$ and $r$ such that
		\bean\label{formula:IWmomentbound7}
		E_{\Sigma_0}(||\hat{\Sigma} ||^{r}) &\le& 
		2^{r-1} E_{\Sigma_0}(||S_n ||^{r}) +  \frac{2^{r-1}}{n^{r}}||A_n||^{r}\nonumber\\
		&\le&     4^{r-1} E_{\Sigma_0}(||S_n-\Sigma_0 ||^{r}) + 4^{r-1}||\Sigma_0||^{r} +  \frac{2^{r-1}}{n^{r}}||A_n||^{r}\nonumber\\
		&\le&  C_5 + C_6\frac{5^k}{n^{r/2}},
		\eean
		for all sufficiently large $n$. The last inequality holds by Lemma \ref{lemma:momentscov}.
		Collecting \eqref{formula:IWmomentbound6} and \eqref{formula:IWmomentbound7}, we have 
		\bea
		\eqref{formula:IWmomentbound1}\le C_7 k^r\Big(  \frac{5^{k/2}}{n^{r/2}} + \frac{5^{3k/2}}{n^r}   \Big),
		\eea
		for some positive constant $C_7$ depending only on $M_0$ and $r$. Combining this inequality with \eqref{formula:IWmomentbound8}, we complete the proof.	 
		
	\end{proof}
	
	\begin{lemma}\label{lemma:IWconcent}
		Suppose the same setting of Lemma \ref{lemma:IWmomentbound}.
		If $(\tau_n-2k)\vee||A_n||\vee k = o(n)$, $||\Sigma_0||\le M_0$ and $x<4||\Sigma_0||$, then there exist some positive constants $C$ and $\lambda$ depending only on $M_0$ such that
		\bea
		E_{\Sigma_0}\{P^\pi(||\Sigma-\Sigma_0||>x\mid \bbZ_n) \} \le C 5^k \{\exp(- \lambda n x^2) + \exp(-\lambda n)\},
		\eea
		for all sufficiently large $n$.
		
	\end{lemma}
	\begin{proof}
		Let $\hat{\Sigma} = (nS_n+A_n)/(n+\tau_n-k-1)$. We have 
		\bea
		E_{\Sigma_0}\{P^\pi(||\Sigma-\Sigma_0||>x \mid \bbZ_n)\} \le E_{\Sigma_0}\{P^\pi(||\Sigma-\hat\Sigma||>x/2\mid \bbZ_n)\} +P_{\Sigma_0}(||\hat\Sigma-\Sigma_0||>x/2).
		\eea
		The upper bound of  $P_{\Sigma_0}(||\hat\Sigma-\Sigma_0||>x/2)$ is given as 
		\bea 
		&&P_{\Sigma_0}(||\hat\Sigma-\Sigma_0||>x/2)\\
		&=&
		P\Big(\Big|\Big|\frac{n}{n+\tau_n-k-1} (S_n-\Sigma_0) + \frac{A_n-(\tau_n-k-1)\Sigma_0}{n+\tau_n-k-1}\Big|\Big|>x/2\Big) \\
		&\le& 
		P(||S_n-\Sigma_0||>x/4) + I\Big( ||A_n||/n + \frac{\tau_n-k-1}{n+\tau_n-k-1}||\Sigma_0|| > x/4\Big)\\
		&\le&P(||\Sigma_0||~||\Omega_{n}^{(k)} - I_k||>x/4) + I\Big( ||A_n||/n + \frac{\tau_n-k-1}{n+\tau_n-k-1}||\Sigma_0|| > x/4\Big)\\
		&=&P\{||\Omega_{n}^{(k)} - I_k||>x/(4||\Sigma_0||)\} + I\Big( ||A_n||/n + \frac{\tau_n-k-1}{n+\tau_n-k-1}||\Sigma_0|| > x/4\Big),
		\eea
		where $\Omega_{n}^{(k)}= \Sigma_0^{-1/2}S_n\Sigma_0^{-1/2}$, of which distribution is $W_k(I_k/n,n)$.
		Since $||A_n|| = o(n)$ and $(\tau_n-k-1)/(n+\tau_n-k-1)=o(1)$, we have
		\bea
		I\Big( ||A_n||/n + \frac{\tau_n-k-1}{n+\tau_n-k-1}||\Sigma_0|| > x/4\Big)=0,
		\eea
		for all sufficiently large $n$.
		Since $x<4||\Sigma_0||$, Lemma 4.4 in \cite{lee2020post} gives
		\bea
		P\{||\Omega_{n}^{(k)} - I_k||>x/(4||\Sigma_0||)\} &\le& 2 \times 5^k \exp\{-n x^2/(2^{11}||\Sigma_0||^2)\}.
		\eea
		Thus, we have 
		\bean
		P_{\Sigma_0}(||\hat\Sigma-\Sigma_0||>x/2) \le  2 \times 5^k \exp\{-n x^2/(2^{11}||\Sigma_0||^2)\}, \label{formula:IWconcent1}
		\eean
		for all sufficiently large $n$.
		
		Next, we show the upper bound of $E_{\Sigma_0}\{P^\pi(   ||\Sigma-\hat{\Sigma}||>x \mid \bbZ_n)\}$. For arbitrary positive constants $R_1$ and $R_2$, we have 
		\bean
		&&E_{\Sigma_0}\{P^\pi(   ||\Sigma-\hat{\Sigma}||>x/2 \mid \bbZ_n)\} \nonumber\\
		&=& 
		E\{P^\pi(   || \hat{\Sigma} ^{1/2}  (\Omega_{n^*}^{(k)} )^{-1}(I_k - \Omega_{n^*}^{(k)} )\hat{\Sigma} ^{1/2}||>x/2 \mid \bbZ_n)  \}  \nonumber\\
		&\le& E\{ P^\pi( ||\hat{\Sigma} || ~ || (\Omega_{n^*}^{(k)} )^{-1}||~  ||I_k -\Omega_{n^*}^{(k)}||>x \mid \bbZ_n) \}\nonumber\\
		&\le& E\{ P^\pi( R_1 ||(\Omega_{n^*}^{(k)})^{-1}||~ ||I_k - \Omega_{n^*}^{(k)}|| >x \mid \bbZ_n)\} + P( ||\hat{\Sigma} ||>R_1)\nonumber\\
		&\le& E\{P^\pi( R_1R_2||I_k - \Omega_{n^*}^{(k)}|| >x\mid\bbZ_n)\}+E[P^\pi\{\lambda_{\min}(\Omega_{n^*}^{(k)}) <R_2^{-1}\mid\bbZ_n\}] \nonumber\\
		&&+ P( ||\hat{\Sigma} ||>R_1),\label{formula:IWconcent2}
		\eean
		where $\Omega_{n^*}^{(k)} = \hat{\Sigma}^{1/2}\Sigma^{-1}\hat{\Sigma}^{1/2}$, of which distribution is $W_k(I_k/(n+\tau_n-k-1),n+\tau_n-k-1)$.
		If we set $R_2^{-1}\le [ 1-\{k/(n + \tau_n-k-1)\}^{1/2}]^2/4$, Lemma 4.2 in \cite{lee2020post} gives that there exists some positive constant $\lambda_1$ depending only on $M_0$ and $M_1$ such that
		\bean
		P^\pi\{\lambda_{\min}(\Omega_{n}^{(k)}) <R_2^{-1}\mid\bbZ_n\} &\le&
		2\exp\{-(n+\tau_n-k-1)[1-\{k/(n+\tau_n-k-1)\}^{1/2}]^2/8\}\nonumber\\
		&\le& 2\exp(-\lambda_1n),\label{formula:IWconcent3}
		\eean
		for all sufficiently large $n$.
		If $R_1 \in (||\Sigma_0||, 3||\Sigma_0||)$, then  
		\bean
		P( ||\hat{\Sigma} ||>R_1)  &\le&
		P( ||\hat{\Sigma}-\Sigma_0 ||>R_1-||\Sigma_0||)\nonumber\\
		&\le& 2 \times 5^k \exp\{-n (R_1-||\Sigma_0||)^2/(2^9||\Sigma_0||^2)\},\label{formula:IWconcent5}
		\eean
		where the second inequality holds by \eqref{formula:IWconcent1}.
		Since we can set $R_2$ large enough to satisfy $x/(R_1 R_2)<1$, Lemma 4.4 in \cite{lee2020post} gives 
		\bean
		P^\pi( R_1 R_2 ||I_k - \Omega_{n^*}^{(k)}|| >x \mid \bbZ_n) 
		&=&
		P^\pi( ||I_k - \Omega_{n^*}^{(k)}|| >x/(R_1 R_2)\mid\bbZ_n) \nonumber\\
		&\le& 2\times 5^k  \exp\Big\{-\frac{(n+\tau_n-k-1)x^2}{2^7(R_1 R_2 )^2}\Big\}.\label{formula:IWconcent6}
		\eean 
		By collecting the inequalities \eqref{formula:IWconcent2}-\eqref{formula:IWconcent6},
		we obtain that there exist some positive constants $C_1$ and $\lambda_2$ depending only on $M_0$ and $M_1$ such that
		\bea
		E_{\Sigma_0}\{P^\pi (   ||\Sigma-\hat{\Sigma}||>x/2 \mid \bbZ_n)  \}    \le 
		C_1 5^k\{\exp(-\lambda_2 nx^2 )+\exp(-\lambda_2 n) \},
		\eea
		for all sufficiently large $n$. Combining this inequality with \eqref{formula:IWconcent1}, we complete the proof.	 
	\end{proof}

	Next, we prove the Lemmas \ref{lemma:upper2}-\ref{lemma:invfactor}.
	
	\begin{proof}[Proof of Lemma \ref{lemma:upper2}]
		By Lemma \ref{lemma:pdadjust} and \ref{lemma:taper}, we have
		\bean
		&&E_{\Sigma_{0}}\{E^{\pi^i}(||T_k^{(\epsilon_n)}(\Sigma)-\Sigma_0||^4\mid\bbZ_n)\}\nonumber\\
		&\le& 2^7E_{\Sigma_{0}}(E^{\pi^i}[||T_k(\Sigma)-\Sigma_0||^4\mid\bbZ_n])+ 2^6\epsilon_n^4 \nonumber\\
		&\le& 2^{10}E_{\Sigma_{0}}\{E^{\pi^i}(||T_k(\Sigma-\Sigma_0)||^4\mid\bbZ_n)\}+
		2^{10}||T_k(\Sigma_0) -\Sigma_0||^4 + 2^6\epsilon_n^4 \nonumber\\
		&\le& 2^{10} 3^4 E_{\Sigma_{0}}\{E^{\pi^i}(\max_{1\le i \le p}||M_i^{(k)} (\Sigma-\Sigma_0) ||^4\mid\bbZ_n)\}+ 2^{10}\{M (\lfloor k/2 \rfloor)^{-\alpha}\}^4+ 2^6\epsilon_n^4.\label{formula:upper2_2}
		\eean
		Next, we have
		\bea
		&&E_{\Sigma_{0}}\{E^{\pi^i}(\max_{1\le i \le p}||M_i^{(k)} (\Sigma-\Sigma_0) ||^4\mid\bbZ_n)\}\\
		&\le& x^4 + E_{\Sigma_{0}}\{E^{\pi^i}(\max_{1\le i \le p}||M_i^{(k)} (\Sigma-\Sigma_0) ||^4 I(\max_{1\le i \le p}||M_i^{(k)} (\Sigma-\Sigma_0) ||>x)\mid\bbZ_n)\}\\
		&\le& x^4 + E_{\Sigma_{0}}\{E^{\pi^i}(\max_{1\le i \le p}||M_i^{(k)} (\Sigma-\Sigma_0) ||^8 \mid\bbZ_n)\}^{1/2} E_{\Sigma_{0}}\{P^{\pi^i}(\max_{1\le i \le p}||M_i^{(k)} (\Sigma-\Sigma_0) ||>x\mid\bbZ_n)\}^{1/2}\\
		&\le& x^4 + p \max_{1\le i \le p} E_{\Sigma_{0}}\{E^{\pi^i}(||M_i^{(k)} (\Sigma-\Sigma_0) ||^8\mid\bbZ_n)\}^{1/2} \max_{1\le i \le p} E_{\Sigma_{0}}\{P^{\pi^i}(||M_i^{(k)} (\Sigma-\Sigma_0) ||>x\mid\bbZ_n)\}^{1/2}.
		\eea
		Note that $M_l^{(k)}(\Sigma_0)\in \calF_{k,\alpha}(M,M_0,M_1)$ and $[M_l^{(k)}(\Sigma)\mid \bbZ_n] \sim IW_k(M_l^{(k)}(nS_n+A_n),n+\nu_n-2p+2k)$ \citep{press2012applied}. If $x<4||M_l^{(k)}(\Sigma_0)||$, then Lemmas \ref{lemma:IWmomentbound} and \ref{lemma:IWconcent} give that there exist some positive constants $C_1$ and $\lambda_1$ depending only on $M_0$ and $M_1$ such that
		\bean
		&&p \max_{1\le i \le p} E_{\Sigma_{0}}\{E^{\pi^i}(||M_i^{(k)} (\Sigma-\Sigma_0) ||^8\mid\bbZ_n)\}^{1/2} \max_{1\le i \le p} E_{\Sigma_{0}}\{P^{\pi^i}(||M_i^{(k)} (\Sigma-\Sigma_0) ||>x\mid\bbZ_n)\}^{1/2} \nonumber\\
		&\le& C_1p  k^4 5^{5k/2} n^{-2} \{\exp(-\lambda_1 nx^2) + \exp(-\lambda_1 n)\},\label{formula:upper2_1}
		\eean
		for all sufficiently large $n$.
		By setting $x=3\log p/(\lambda_1 n) $, which is smaller than an arbitrary positive constant for sufficiently large $n$, We show that \eqref{formula:upper2_1} is bounded above by some positive constant for all sufficiently large $n$. Since $\epsilon_n=O(1)$, \eqref{formula:upper2_2} is bounded by some positive constant depending only on $M_0$, $M_1$, $M$ and $\alpha$.	 
		
	\end{proof}

	\begin{proof}[Proof of Lemma \ref{lemma:evmin}]

		By inequality \eqref{formula:evminfreq1}, we have
		\bea
		I[\lambda_{\min}\{T_k(\Sigma)\} \le  c]
		&\le&
		I( ||T_k(\Sigma-\Sigma_0)|| \ge c_k )
		\eea
		where $c_k=\lambda_{\min}(\Sigma_0) - M(\lfloor k/2\rfloor )^{-\alpha} -c$. Since $c \le \lambda_{\min}(\Sigma_0)/2$ and $\lfloor k/2\rfloor> \{4M/\lambda_{\min}(\Sigma_0)\}^{1/\alpha}$, 
		we have $c_k\ge  \lambda_{\min}(\Sigma_0)/4$. 
		There exist some positive constants $C$ and $\lambda$ depending only on $M_0$ and $M_1$ such that
		\bea
		E_{\Sigma_0}(P^{\pi^i}[\lambda_{\min}\{T_k(\Sigma)\} \le  c \mid\bbZ_n] ) &\le&
		E_{\Sigma_0}\{P^{\pi^i}(\lambda_{\min}( ||T_k(\Sigma-\Sigma_0)|| \ge c_k \mid\bbZ_n)\} \\
		&\le&   p \max_{1 \le l\le p}E_{\Sigma_0}\{P^{\pi^i}( ||M_l^{(k)}(\Sigma-\Sigma_0) ||   \ge c_k/3 \mid\bbZ_n)\}\\
		&\le&    p C 5^k [\exp\{-\lambda n (c_k/3)^2\} + \exp(-\lambda n)],
		\eea
		for all sufficiently large $n$. The second inequality is satisfied by Lemma \ref{lemma:taper}. 
		For the third inequality, Lemma \ref{lemma:IWconcent} is used. Note  $M_l^{(k)}(\Sigma_0)\in \calF_{k,\alpha}(M,M_0,M_1)$, $[M_l^{(k)}(\Sigma)\mid \bbZ_n] \sim IW_k\{M_l^{(k)}(nS_n+A_n),n+\nu_n-2p+2k\}$ and 
		$c_k/3$ satisfies
		\bea
		c_k/3 &\le& \lambda_{\min}(\Sigma_0)\\
		&\le& \lambda_{\min}\{M_l^{(k)}(\Sigma_0)\}\\
		&\le& 4||M_l^{(k)}(\Sigma_0)||,
		\eea
		for all $l\in \{1,2,\ldots, p\}$.	 
	\end{proof}


	\begin{proof}[Proof of Lemma \ref{lemma:TPPPupper}]
		By Lemmas \ref{lemma:pdadjust} and \ref{lemma:taper}, we have
		\bean\label{formula:TPPPupper1}
		&&E_{\Sigma_0}\{ E^{\pi^i} ( ||\Sigma_0 - T_k^{(\epsilon_n)}(\Sigma) ||^2 \mid\bbZ_n)  \}\nonumber\\ 
		&\le& 	2^3 E_{\Sigma_0}\{ E^{\pi^i} ( ||\Sigma_0 - T_k(\Sigma) ||^2 \mid\bbZ_n)  \} + 4 \epsilon_n^2\nonumber\\
		&\le& 	2^4 E_{\Sigma_0}\{ E^{\pi^i} ( ||T_k(\Sigma_0) - T_k(\Sigma) ||^2 \mid\bbZ_n)  \} + 2^4||T_k(\Sigma_0) - \Sigma_0 ||^2+ 4 \epsilon_n^2\nonumber\\
		&\le& 	2^4 3^2 E_{\Sigma_0}\{ E^{\pi^i} ( \max_{1 \le l\le p}||M_l^{(k)}(\Sigma_0-\Sigma) ||^2 \mid\bbZ_n)  \} + 2^4 (M \lfloor k/2\rfloor^{-\alpha})^{2}+ 4 \epsilon_n^2.
		\eean
		Note  $M_l^{(k)}(\Sigma_0)\in \calF_{k,\alpha}(M,M_0,M_1)$, $[M_l^{(k)}(\Sigma)\mid \bbZ_n] \sim IW_k\{M_l^{(k)}(nS_n+A_n),n+\nu_n-2p+2k\}$. Thus, for an arbitrary positive real number $x$ with $x<4\min_{1\le l\le p}||M_l^{(k)}(\Sigma_0)||$,
		there exist some positive constants $C_1$ and $\lambda_1$ depending only on $M_0$ such that
		\bea
		&&E_{\Sigma_0}\{ E^{\pi^i} ( \max_{1 \le l\le p}||M_l^{(k)}(\Sigma_0-\Sigma) ||^2 \mid\bbZ_n)  \} \\
		&\le& x^2 + E_{\Sigma_0}[E^{\pi^i}\{ \max_{1 \le l\le p}||M_l^{(k)}(\Sigma_0-\Sigma) ||^2I( \max_{1 \le l\le p}||M_l^{(k)}(\Sigma_0-\Sigma) || >x)\mid\bbZ_n\}  ]\\
		&\le&x^2 + p \max_{1 \le l\le p} E_{\Sigma_0}\{E^{\pi^i} (||M_l^{(k)}(\Sigma_0-\Sigma) ||^4 \mid\bbZ_n    )\}^{1/2} 
		\max_{1 \le l\le p} E_{\Sigma_0}\{E^{\pi^i}( ||M_l^{(k)}(\Sigma_0-\Sigma) || >x  \mid\bbZ_n    )\}^{1/2} \\
		&\le& x^2 + C_1pk^2 \frac{5^3}{n}  5^k \{\exp(-\lambda_1 n x^2)+ \exp(-\lambda_1 n)\},
		\eea
		for all sufficiently large $n$. The last inequality holds by Lemmas .
		By setting $x^2=  2(\log p + k\log 5)/n\lambda_1$, which is lower than an arbitrary positive constant for all sufficiently large $n$, we have 
		\bea
		E_{\Sigma_0}\{ E^{\pi^i} ( \max_{1 \le l\le p}||M_l^{(k)}(\Sigma_0-\Sigma) ||^2 \mid\bbZ_n)  \} \le C_2 \Big( \frac{\log p+k}{n}   \Big),
		\eea
		for some positive constants $C_2$ depending only on $M_0$. 
		Combining this inequality with \eqref{formula:TPPPupper1}, we complete the proof.	 
	\end{proof}

	\begin{proof}[Proof of Lemma \ref{lemma:invfactor}]
		Let $c=\lambda_{\min}(\Sigma_0)/2$. Since $\lambda_{\min}\{T_k^{(\epsilon_n)}(\Sigma)_{11}\}\ge \lambda_{\min}\{T_k^{(\epsilon_n)}(\Sigma)\}\ge \epsilon_n$, we have
		\bean
		&&E_{\Sigma_0}\{ E^{\pi^i}(|| \Sigma_{0,11}^{-1} - T_k^{(\epsilon_n)}(\Sigma)_{11}^{-1}||^2\mid \bbZ_n)\} \nonumber \\
		&\le& ||\Sigma_{0}^{-1}||^2 E_{\Sigma_0}\{ E^{\pi^i}(||T_k^{(\epsilon_n)}(\Sigma)_{11}^{-1}||^2~|| \Sigma_{0,11} - T_k^{(\epsilon_n)}(\Sigma)_{11}||^2 \mid \bbZ_n)\} \nonumber\\
		&\le& \frac{1}{M_1^2c^2}E_{\Sigma_0}\{ E^{\pi^i}(|| \Sigma_{0,11} - T_k^{(\epsilon_n)}(\Sigma)_{11}||^2 \mid \bbZ_n)\} \nonumber\\
		&&+\frac{1}{M_1^2\epsilon_n^2}E_{\Sigma_0}\{ E^{\pi^i}(|| \Sigma_{0,11} - T_k^{(\epsilon_n)}(\Sigma)_{11}||^2 I[\lambda_{\min}\{T_k(\Sigma)\}\le c ] \mid \bbZ_n)\} \nonumber\\
		&\le& \frac{1}{M_1^2c^2}E_{\Sigma_0}\{ E^{\pi^i}(|| \Sigma_{0,11} - T_k^{(\epsilon_n)}(\Sigma)_{11}||^2 \mid \bbZ_n)\} \label{formula:invfactor1}\\
		&&+\frac{1}{M_1^2\epsilon_n^2}E_{\Sigma_0}\{ E^{\pi^i}(|| \Sigma_{0,11} - T_k^{(\epsilon_n)}(\Sigma)_{11}||^4\mid\bbZ_n)\}^{1/2}  E_{\Sigma_0}\{E^{\pi^i}(I[\lambda_{\min}\{T_k(\Sigma)\}\le c ] \mid \bbZ_n)\}^{1/2} .\label{formula:invfactor3}
		\eean
		For the upper bound of \eqref{formula:invfactor1}, Lemma \ref{lemma:TPPPupper} gives
		\bea
		E_{\Sigma_0}\{ E^{\pi^i}(|| \Sigma_{0,11} - T_k^{(\epsilon_n)}(\Sigma)_{11}||^2 \mid \bbZ_n)\} &\le&	E_{\Sigma_0}\{ E^{\pi^i}(|| \Sigma_{0} - T_k^{(\epsilon_n)}(\Sigma)||^2 \mid \bbZ_n)\} \\
		&\le& C_1\Big( \frac{k+\log p}{n} + k^{-2\alpha} +\epsilon_n^2 \Big),
		\eea
		for some positive constant $C_1$ depending only on $M$, $M_0$, $M_1$ and $\alpha$.
		Lemmas \ref{lemma:upper2} and \ref{lemma:evmin} gives that there exist positive constants $C_2$ and $\lambda$ depending only on $M$, $M_0$, $M_1$ and $\alpha$ such that    
		\bea
		\eqref{formula:invfactor3}&\le&
		C_2 \frac{p^{1/2} 5^{k/2}}{\epsilon_n^2} \exp(-\lambda_1 n) ,   
		\eea
		for all sufficiently large $n$. 
		Collecting the upper bound of  \eqref{formula:invfactor1} and \eqref{formula:invfactor3}, we complete the proof.

	\end{proof}

\bibliography{cov-ppp}